\documentclass[reqno]{amsart}
\usepackage[utf8]{inputenc}
\usepackage[T1]{fontenc}
\usepackage{enumerate}
\usepackage{lmodern, microtype}
\usepackage{supertabular, empheq}

\usepackage{amssymb, amsmath, amsthm, xcolor}
\usepackage{tikz}
\usetikzlibrary{decorations.pathreplacing}

\definecolor{darkgreen}{rgb}{0,0.5,0}
\definecolor{darkred}{rgb}{0.7,0,0}
\definecolor{darkblue}{rgb}{0,.2,.7}
\definecolor{darkyellow}{rgb}{1, 0.75, 0}
\definecolor{Ggreen}{RGB}{76,153,0}

\usepackage[colorlinks, 
citecolor=darkred, linkcolor=darkgreen, urlcolor=darkblue,
 pdfpagelabels=true,
 unicode=true,
]{hyperref}

\hypersetup{
 pdfauthor={Nadine Große and Niccolo Pederzani},
 pdftitle={Homotopy equivalence of spaces of metrics with invertible Dirac operator},
 pdflang={en}
}
\usepackage[all]{hypcap}

\theoremstyle{plain}      
\newtheorem{thm}{Theorem}
\newtheorem{theorem}[thm]{Theorem}    
\newtheorem*{theorem*}{Theorem}    
\newtheorem{corollary}[thm]{Corollary}          
\newtheorem{lemma}[thm]{Lemma}
\newtheorem{proposition}[thm]{Proposition}     

\theoremstyle{definition}      
\newtheorem{definition}[thm]{Definition}

\newtheorem{remark}[thm]{Remark}

\makeatletter
\newtheorem*{rep@theorem}{\rep@title}
\newcommand{\newreptheorem}[2]{%
\newenvironment{rep#1}[1]{%
 \def\rep@title{#2 \ref{##1}}%
 \begin{rep@theorem}}%
 {\end{rep@theorem}}}
 \makeatother
 \newreptheorem{definition}{Definition}

\newcommand{\define}{\textrm{:=}}

\renewcommand{\ker}{\text{ker}\,}
\newcommand{\dvol}{\text{dvol}}

\newcommand{\R}{\mathcal{R}}

\newcommand{\Riem}{\mathrm{Riem}(M)}
\newcommand{\Ri}{\mathrm{Riem}(M\setminus S)}
\newcommand{\RM}{\mathcal{R}^\text{inv}(M)}
\newcommand{\Riemf}{{\mathrm{Riem}}^{\text{inv}}_{\frac{1}{2}\text{flat}}(M)}
\newcommand{\Rf}{\mathcal{R}^{\text{inv}}_{\frac{1}{2}\text{flat}-D_1}(M)}
\newcommand{\Rc}{\mathcal{R}^{\text{inv}}_{cyl}(M\setminus S)}

\newcommand{\Rcrev}{\mathcal{R}^{\text{inv}}_{cyl}(\widetilde{M}\setminus \widetilde{S})}
\newcommand{\Riemcu}{{\mathrm{Riem}}_{cyl}(M\setminus S)}
\newcommand{\adh}{S}

\author[N. Gro\ss e]{Nadine Gro{\ss}e} \address{N. Gro{\ss}e, Mathematisches Institut,
 Universit\"at Freiburg, 79104 Freiburg, Germany}
\email{nadine.grosse@math.uni-freiburg.de}

\author[N. Pederzani]{Niccol{\`o} Pederzani} \address{N. Pederzani, Liceo Scientifico
Attilio Bertolucci, Parma, Italy}
\email{niccolo.pederzani88@gmail.com}

\begin{document}

\title[Homotopy equivalence of spaces of metrics with invertible Dirac]{Homotopy equivalence of spaces of metrics with invertible Dirac operator}
\begin{abstract}
We prove that for  cobordant closed spin manifolds of dimension $n\geq 3$ the associated spaces of metrics with invertible Dirac operator are homotopy equivalent. This is the spinorial counterpart of a similar result on positive scalar curvature of Chernysh/Walsh and generalizes the surgery result of Ammann-Dahl-Humbert on the existence of metrics with invertible Dirac operator under surgery. We also give a relative statement of this homotopy equivalence.
 \end{abstract}

\maketitle
\section{Introduction}

Let $(M,g)$ be a closed spin Riemannian manifold $M$ of dimension $n$. Let $D^g$ be the associated (classical) Dirac operator.\medskip 

From the Atiyah-Singer index theorem it is known that the index of the Dirac operator on a closed manifold is a differential topological invariant of the manifold. This leads to a lower bound of the dimension of the kernel of the Dirac operator depending on the dimension $n$, \cite[Sec.~3]{BD}, 
\begin{align*}
 \dim \ker D^g \geq \left\{  \begin{matrix} 
                              |\hat{A}(M)| & n\equiv 0 \text{ mod }4\\
                              1 &n\equiv 1 \text{ mod }4, \alpha(M)\neq 0\\
                             2&n\equiv 2 \text{ mod }4,  \alpha(M)\neq 0\\
                             0 &\text{else}, 
                             \end{matrix}
\right.
\end{align*}
where the $\hat{A}$-genus and the $\alpha$-genus are invariants of the spin bordism class of $M$.\medskip 

The question of (non)existence of metrics with harmonic spinors is related to the question of existence of metrics with positive scalar curvature via the Schrödinger-Lichnerowicz formula 
\[ (D^g)^2 =  \nabla^*\nabla +\frac{\text{scal}_g}{4},
\]
where $\nabla$ is the lifted Levi-Civita connection on the spinor bundle. More precisely, if $g$ has positive scalar curvature, then $D^g$ has to be invertible as an operator on the Hilbert space of $L^2$-spinors to itself. In particular, the space $\mathcal{R}^{\text{pos}}(M)$ of metrics on $M$ with positive scalar curvature is a subset of the space $\RM$ of metrics on $M$ with invertible Dirac operator.\label{page:Rinv}\medskip 

  \begin{figure}
\begin{tikzpicture}[scale=0.59]

\draw[line width=1] plot [smooth cycle] coordinates {(0,0) (1,1) (3,1) (5,-1) (3,-3)(0,-2)};
\draw[line width=1] (1.5,-1) .. controls (2.3,-2) .. (3.3,-1.5);
\draw[line width=1] (1.8,-1.4) .. controls (2.3,-1.) .. (3,-1.6);
\draw (0.5,-3) node {$M$};
\draw[->] (7,2.3) .. controls (7.5,1.7) .. (4,0);
\draw (5.5,1.1) node {$\iota$};

\draw[->] (6,-0.8)--(11,-0.8);
\draw (8.3, -0.3) node {cut out image$(\iota)$};
\draw (8.4, -1.3) node {glue in $S^{n-k-1}\times D^{k+1}$};

\draw (6.8, 2.6) node {$D^{n-k}\times S^k$};

\draw[fill, opacity=0.2]  (2,0) circle (0.3cm);
\draw (2,0) circle (0.03cm);

\draw[fill, opacity=0.2]  (3.5,-0.5) circle (0.3cm);
\draw (3.5,-0.5) circle (0.03cm);

\draw[line width=1, xshift=12cm] plot [smooth cycle] coordinates {(0,0) (1,1) (3,1) (5,-1) (3,-3)(0,-2)};
\draw[line width=1, xshift=12cm] (1.5,-1) .. controls (2.3,-2) .. (3.3,-1.5);
\draw[line width=1, xshift=12cm] (1.8,-1.4) .. controls (2.3,-1.) .. (3,-1.6);
\draw[ xshift=12cm] (0.5,-3) node {$\widetilde{M}$};

\draw[fill, color=white, xshift=12cm](1.7,0) .. controls (1.7, 2.2) and (4.4,2.5) .. (3.8,-0.5)  --
(3.15,-0.5) .. controls (4.4,1.6) and (1.9, 1.5)  .. (2.3,-0.1)-- (1.7,0);

\draw[xshift=12cm, line width=1] (1.7,0) .. controls (1.7, 2.2) and (4.4,2.5) .. (3.8,-0.5);
\draw[xshift=12cm, line width=1] (2.3,-0.1) .. controls (1.9, 1.5) and (4.4,1.6) .. (3.15,-0.5);
\draw[xshift=12cm, dashed] (3.2,1.34) ellipse [x radius=0.15cm,y radius=0.33cm];

\draw[xshift=12cm, dashed]  (2,0) circle (0.3cm);
\draw[xshift=12cm, dashed]  (3.5,-0.5) circle (0.3cm);

\end{tikzpicture}
\caption{$\widetilde{M}$ arises from $M$ by a \emph{surgery of dimension $k$ (= codimension $n-k$)} along the \emph{surgery sphere} $S\define \iota(\{0\}\times S^k)$ where $\iota$ is an embedding of the $(n-k)$-disk $D^{n-k}$ times the $k$-sphere $S^k$. (The picture is for $k=0$ and $n=2$). This surgery is called \emph{spin-sugery} if the $S^{n-k-1}$ factor (in the part that is glued in) is equipped with the bounding spin structure, i.e. the spin structure obtained by restricting the unique spin structure on $D^{n-k}$. For $n-k > 2$ this spin structure is unique anyway (we always fix the orientations).}
\label{fig:surg} 
\end{figure}
This relation in mind, there might be a chance to obtain theorems/constructions known in the setting of positive scalar curvature also for the invertibility of Dirac operators. For the surgery result of Gromov-Lawson for positive scalar curvature this was done by Ammann, Dahl and Humbert in \cite{ADH}. More precisely, let there be  an embedded sphere $S^k$ in $M$ with trivial normal bundle $\nu_{S^k}\cong D^{n-k}\times S^k$ (we choose the trivialization once and for all). Here, $D^\ell$ is the $\ell$-dimensional unit disk. Then a surgery of codimension $n-k$ corresponds topologically to setting $$\widetilde M=M\setminus \iota\left(D^{n-k}\times S^k\right)\underset{S^k\times S^{n-k-1}}{\cup}\left( S^{n-k-1}\times D^{k+1} \right),$$
cp. Figure~\ref{fig:surg}. In \cite{GL} Gromov and Lawson showed that having a metric $g$ on $M$ with positive scalar curvature and an $\widetilde{M}$ obtained from  $M$ via a surgery of codimension $n-k\geq 3$ there is always a metric on $\widetilde{M}$ with positive scalar curvature. A similar statement for the kernel of the Dirac operator, but using different techniques, was obtained in \cite{ADH} by Ammann, Dahl and Humbert: If $\widetilde{M}$ is obtained from $M$ via a spin surgery of codimension $n-k\geq 2$, then for every Riemannian metric $g$ on $M$ there is a metric $\tilde{g}$ on $\widetilde{M}$ with 
\[ \dim\, \ker\, D^{\tilde{g}}\leq \dim\, \ker\, D^{{g}}.\]
For an intuition why here a lower codimension can be assumed see Section~\ref{sec:surg}. Using bordism techniques and sufficient knowledge on enough model manifolds, this result  implies that on all connected spin manifolds there is a metric $g$ such that the lower bound on $\dim\, \ker\, D^g$ is attained, \cite[Thm. 4.1]{ADH}. See also \cite[Thm. 3.9]{BD} for the case $n\geq 5$.\medskip

The Gromov-Lawson result  was generalized to a statement on the homotopy type of the space of metrics with positive scalar curvature by  V. Chernysh and M. Walsh. For that $\mathcal{R}^{\text{psc}}(M)$ is equipped with the compact-open $C^\infty$-topology. 

\begin{theorem*}\cite{Chernysh, Walsh}
 Let $M^n$ and $\widetilde M$ be two closed manifolds of dimension $n$ obtained one another via a sequence of surgery transformations of dimension $2\leq k \leq n-3$. Then the corresponding spaces of Riemannian metrics with positive scalar curvature have
  the same homotopy type: $$\mathcal{R}^{\text{psc}}(M)\simeq\mathcal{R}^{\text{psc}} (\widetilde M).$$
\end{theorem*}

This result and a generalization to families of Morse functions \cite{WI,WII}  underlie a lot of topological applications on the homotopy type of $\mathcal{R}^{\text{psc}}(M)$, see e.g. \cite{BHSW, HSS, BERW}. See also \cite{R07, S14,T16} for some surveys covering related topics.\medskip

The goal of this article is to obtain an analogous result for the space $\mathcal{R}^{\text{inv}} (M)$. In this case we expect the range of surgeries allowed to be $1\leq k\leq n-2$, which covers all the possible surgeries needed to connect two spin cobordant connected manifolds. We show

  \begin{theorem}\label{thm:main}
 Let $M^n$ and $\widetilde M^n$ be two  closed connected nonempty spin manifolds of dimension $n\geq 3$. If $\widetilde M$ is spin cobordant to $M$, then $$\mathcal{R}^{\text{inv}}(M)\simeq\mathcal{R}^{\text{inv}}(\widetilde M).$$
\end{theorem}

This raises a lot of follow-up questions that are not addressed here but  which are worth to investigate further:  Does the homotopy equivalence depend on the chosen Morse function? Do there exist similar results for other Dirac-type operators?

\subsection*{Structure of the article} In Section 2 we give the necessary underlying analytical results on the Dirac operator. The actual proof of Theorem~\ref{thm:main} starts in Section~\ref{sec:3}. There we lay out all the steps to prove the corresponding result where only one surgery is involved--this is Proposition~\ref{prop:main2}. Section~\ref{sec:3} contains the complete strategy of the proof of Proposition~\ref{prop:main2}. However, we outsource longer proofs and constructions of auxiliary results in between to Sections~\ref{pr:Riemf} to~\ref{sec:cup}. The coarse structure of this proof is again summarized in Figure~\ref{fig:not} and a short table of the most important notations are given in Table~\ref{tab1}. In Subsection~\ref{sec:rel} we also give a relative version of Proposition~\ref{prop:main2}.

\subsection*{Acknowledgement} We would like to thank Bernd Ammann for a many helpful
discussions. Several steps of this proof were worked out by the second author in his PhD-thesis \cite{Ped}.

\section{Preliminaries}

We assume that  $(M,g)$ is spin and that the spin structure is chosen once and forever. We always use the compact-open-$C^\infty$-topology for metrics and functions on $M$. The space of all Riemannian metrics on $M$ will be denoted by $\Riem$.\medskip 

Let  $\Sigma^g M$ denote the associated complex spinor bundle.  We denote by  $\nabla^g$ and $D^g$ the lift of the Levi-Civita connection to the spinor bundle and the Dirac operator, respectively.

\subsection{On spin surgery}\label{sec:surg}

Let $S\cong S^k\times \{0\} \subset S^k\times D^{n-k}\hookrightarrow M$ be the surgery sphere of a spin surgery of codimension $n-k$, and let $\widetilde{M}$ be the smooth manifold after this surgery, see Figure~\ref{fig:surg}. Then, $\widetilde{S}\cong \{0\}\times S^{n-k-1}\subset D^{k+1}\times  S^{n-k-1}\subset \widetilde{M}$ is the surgery sphere of the reverse surgery which then has codimension $k+1$.\medskip 

We note that in the case $n-k=2$ the surgery sphere $\widetilde{S}$ has as induced spin structure the one (unique up to orientation) that bounds the disk (i.e., coming from $D^2$). Hence, the Dirac operator on $\widetilde{S}\cong S^{n-k-1}$ w.r.t. the standard metric  is always invertible.\medskip 

There is a difference to the positive scalar curvature case: If we equip $\mathbb R^{k+1}\times S^{n-k-1}$ with the standard product metric, then the scalar curvature is positive only for $n-k\geq 3$ but the Dirac operator is invertible for $n-k\geq 2$. This gives an intuition where the different codimensions in the surgery results for $\mathcal{R}^{\text{psc}}$ and $\mathcal{R}^{\text{inv}}$ come from.\medskip

Similarly, in order to be able to glue in $D^{k+1}\times  S^{n-k-1}$ during the surgery the spin structure on $S\cong S^k$ needs to be the one that bounds the disk.

\subsection{Identification of spinor bundles to different metrics}
The spinor bundle depends on the metric. In order to compare spinors to different metrics we use the identification of spinor bundles as in \cite[Sec. III]{BG}, see also \cite[Sec. 2.1]{ADH}: \medskip

The change of the spinor bundle when changing the metric on $M$ from $g$ to some  metric $h$ is given by a canoncical identification map: 
\begin{align}\label{eq:beta}
 \beta^g_h\colon \Sigma^g(M)\to \Sigma^h(M)
\end{align}
that is fiberwise an isometry with $\beta^g_h=(\beta^h_g)^{-1}$.\medskip

The Dirac operator $D^h$ can be expressed in terms of $D^g$ via these maps. The resulting Dirac operator on $(M,h)$ is related to $D^h$ \begin{equation}\label{Dg1g2AB}{}^hD^g= D^h+A^h_g\circ\nabla^h+B^h_g,\end{equation} 
where 
\begin{equation}\label{Dg1g2}^hD^g\define \beta_h^g\circ D^g\circ\beta^h_g\end{equation} and where $A^h_g\in\Gamma\left(TM\otimes \mathrm{End}(\Sigma^{h}M)\right)$ and $B^h_g\in\Gamma\left(\mathrm{End}(\Sigma^hM)\right)$. The latter sections  satisfy the following inequalities in terms of the $C^0$-norm of the section $g-h\in\Gamma(T^* M\otimes T^* M)$:
\begin{equation}\label{spinorisonorm}
 |A^h_g|\leq C|g-h|_h,\qquad |B^h_g|\leq C(|g-h|_h+|\nabla^h(g-h)|_h)
\end{equation}
for some $C>0$.\medskip 

In particular, for a conformal change of the metric $h=F^2g$, the Dirac operators $D^g$ and $D^{h}$ are related by 
\begin{equation}\label{conformalchange}
 F^\frac{n+1}{2}D^h=D^gF^{\frac{n-1}{2}},
\end{equation}
see \cite[Sec. 1.4]{Hi} (for conformal metrics we suppress the identification maps $\beta$ in the notations).

\subsection{Manifolds with cylindrical ends}

Let $(N,h)$ be a Riemannian manifold. We assume that there is a compact subset $K\subset N$ such that $N\setminus K$ is diffeomorphic to $Z\times [0,\infty)$ for some closed manifold $Z$. Note that $Z$ does not have to be connected. For a connected component $(Z_i, h_i\define h|_{Z_i\times \{0\}})$ of $(Z,h)$, let $h$ on $Z_i\times [0,\infty)$ have the form $h_i+du^2$. We then call $(Z_i\times [0,\infty), h)$ a \emph{cylindrical end}. If for all connected components $Z_i$ of $Z$ the manifold $(Z_i\times [0,\infty), h)$ is a cylindrical end, we call $(N,h)$ a \emph{manifold with cylindrical ends}.\medskip 

Such manifolds with cylindrical ends are in particular complete. Hence, the Dirac operator  
 $D^g$ for a manifold with cylindrical ends is essentially self-adjoint when considered as an unbounded operator from $L^2(\Sigma^gM)$ to itself.\medskip 
 
 Next we collect some spectral properties of manifolds with cylindrical ends:
 
\begin{lemma}\label{lem:spec_cyl}
Let $(N,h)$ be a Riemannian manifold with cylindrical ends $(Z\times [0,\infty), \hat{h}+du^2)$. Then, the following hold:
\begin{enumerate}[(i)]
 \item \cite[Sec. 4]{Mu} If the Dirac operator $D^{\hat{h}}$ on $Z$ has a spectral gap around zero, then the essential spectrum of the Dirac operator $D^h$ on $N$ has a gap around zero.
 \item \cite[Prop. 6.1 - Lem. 6.3]{DG} If the metric only changes on a compact subset of $(N,h)$, then the infimum of the spectrum of Dirac squared depends continuously on $g$.
\end{enumerate}
\end{lemma}

\subsection{Regularity results}

Here we collect some regularity results for spinors we need in the following:

\begin{lemma}\label{lem:removalsingularityF} (Removal of singularities, \cite[Lem. 2.4]{ADH})
Let $(M, g)$ be a Riemannian manifold, let $S\subset M$ be a compact submanifold of dimension $k\leq n-2$.  Let  $\phi\in L^2(\Sigma^{g} M, g)$ fulfil $ D^{g}\phi =0$ weakly on $M\setminus S$.  Then,  $ D^{g}\phi =0$ holds weakly on $M$ and, hence, $\phi$ is smooth.
\end{lemma}

\begin{lemma}\label{lem:Schauder}(Parametrized version of Schauder in \cite[Lem. 2.2]{ADH})
 For given $g\in \Riem$ and $K\subset M$ compact, there is a neighbourhood $U\subset \Riem$ of $g$ and a constant $C=C(K,M,g)$ such that 
 for all $h\in U$ and all harmonic spinors $\psi$ on $(M,h)$ 
 \[ \Vert \beta_g^h\psi\Vert_{C^2(K,g)}\leq C\Vert \beta^h_g\psi\Vert_{L^2(M,g)}.\]
\end{lemma}

\begin{proof} The proof is as in \cite[Lem. 2.2]{ADH}.\end{proof}

\begin{lemma}\label{lem:Arzela}(Arz\'ela-Ascoli, \cite[Thm. 1.33]{Adams})
Let $K$ be a compact subset of a Riemannian manifold $(M,g)$, let $\phi_i$ a bounded sequence in $C^{1,\alpha}(\Sigma^gK, g)$ for some $\alpha >0$. Then, a subsequence of $\phi_i$ converges in $C^1(\Sigma^gK, g)$.
\end{lemma}

\section{Steps to prove Theorem~\ref{thm:main}}\label{sec:3}

In this section we lay out the steps to prove a surgery result that is the $\mathcal{R}^{\text{inv}}(M)$ counterpart to the positive scalar curvature case from \cite{Chernysh, Walsh}:

  \begin{proposition}\label{prop:main2}
 Let $M^n$ and $\widetilde M^n$ be two  closed spin manifolds of dimension $n\geq 3$ where $\widetilde{M}$ can be obtained from $M$ by a spin surgery of codimension $2\leq n-k\leq n-1$. Then $\mathcal{R}^{\text{inv}}(M)$ and $\mathcal{R}^{\text{inv}}(\widetilde M)$ are homotopy equivalent.
\end{proposition}

In this section we will explain the main points of the proof of this result. Longer proofs and constructions of auxiliary results in between are outsourced to Sections~\ref{pr:Riemf} to~\ref{sec:cup}. 
In Section~\ref{sec:mainthm} we will see how this result implies via standard bordism arguments our main theorem.\medskip 

Actually, we will also obtain a relative statement of Proposition~\ref{prop:main2},  since given any compact subset $A\subset M$ such that $M\setminus A$ contains the surgery sphere the constructions can be carried out such the metric does not change on $A$, see Proposition~\ref{prop:main3}.\medskip 

Very broadly speaking, in order to obtain Proposition~\ref{prop:main2} the first  idea is to mimic the proof of \cite[Thm. 1.2]{ADH}---the spinorial analogue of the Gromov-Lawson result---in a parametrized way: there the authors changed first a fixed metric $g\in \RM$ such that it has a standard form near the surgery sphere $S$: $(\text{flat metric}) + g|_S$. Then, using a conformal change that goes with $1/d_g(.,S)$ on an annulus of $S$ and also changing the metric in the $S^k$-direction near $S$, they obtain a 'blown-up metric'. That is an invertible metric on $M$ with a standard cylindrical end $([0,\infty)\times S^{n-k-1}\times S^k, du^2+\sigma_{n-k-1}+\sigma_k)$ and a 'torpedo', that represents the surgery, grafted on the end. Here, $\sigma_\ell$ denotes the standard metric on $S^\ell$, and $d_g$ is the distance function w.r.t. $g$.\medskip 
 
For doing the above in a parametric way we separate the ad-hoc topology changes  when glueing in the torpedo from the blowing-up procedure. Hence, we want to blow up in a parametric way just to the standard cylindrical end, i.e., in particular to a standardized metric on $M\setminus S$, see also Section~\ref{step1} below.\medskip  

  In order to carry out this idea in more details, let from now on $g_0$ be \emph{a fixed background metric} on $M$ specified further below. As stated before, we always use the compact-open-$C^\infty$-topology for metrics and functions both on $M$ and on $M\setminus S$. For the actual estimates we use the  $C^\infty$-norm of functions and metrics on $M\setminus S$ w.r.t. the background metric $g_0$.

\subsection{It is sufficient to prove \texorpdfstring{$\RM\cong \Rc$}{RM equivalent to Riemc} for \texorpdfstring{$n-k\geq 2$}{n-kgeq2}:}\label{step1}

We choose the {fixed background metric} $g_0$ on $M$ such that 
\[ \exp_{g_0}^\perp\colon (\overline{B_2(0)}\subset \mathbb R^{n-k})\times S\to M\]
is a diffeomorphism onto its image. Here, $\overline{B_2(0)}$ is the closed $2$-ball in $\mathbb R^{n-k}$ around the origin and  $\exp_{g_0}^\perp$ is the normal exponential map to $S$ in $(M,g_0)$. We set  $K\define \exp_{g_0}^\perp(\overline{B_2(0)}\times S)\subset M$ and $D_r\define \{ p\in M\ |\ d_{g_0}(p,S)\leq r\}$ for all $r\in (0,2]$, and we identify $\{r=0\}$ with $S$. In the following, 
\begin{center}\emph{$r$ will  always be the radial coordinate to $S$ w.r.t $g_0$}.\end{center}

We define $\Riemcu\subset \Ri \times (0,1) $ to consist of all $(g,s)\in \Ri\times (0,1]$ with the following properties:
\begin{enumerate}[(I)]\label{page:standard_form}
 \item $g$ has a \emph{cylindrical end w.r.t. $\ln r$} starting at $s$, i.e.,  $D_s\setminus S$ is isometric to $\frac{dr^2}{r^2} +\sigma_{n-k-1}+\sigma_k$.
 \item $g$ has \emph{standard form} on $D_1\setminus S$, i.e., there is a smooth function $z\colon (0,1]\to \mathbb R_{>0}$ and smooth families $g^i(r)$ of metrics on the sphere $S^i$  such that $g=z(r)^2 dr^2 + g^{n-k-1}(r)+ g^k(r)$ on $D_1\setminus S$. 
\end{enumerate}

Note that (I), the cylindrical end, is the important property for what follows. Property (II) is mainly for convenience and makes it easier to write down some maps in Section~\ref{step6_1}. The notation 'w.r.t. $\ln r$'  refers to the fact that putting $u=-\ln r$ the metric on $D_s\setminus S$ has the usual form $du^2+\sigma_{n-k-1}+\sigma_k$ for $u\in (-\ln s, \infty)$.\medskip

  We define\label{page:inv_cyl} $\Rc\subset \Riemcu$ to contain all $(g,s)\in \Riemcu$ for which the Dirac operator $D^g$ is invertible.\medskip 
 
  Further we equip $\Ri\times (0,1)$ with the distance function \[d_{\R}((g_1,s_1),(g_2,s_2))\define \Vert g_1-g_2\Vert_{C^\infty(M\setminus S, g_0)} + |s_1-s_2|.\]  This makes $\Ri\times (0,1]$, and hence its subspace $\Rc$, into a metric space.   
    \begin{remark}\label{rem:R_metric}\hfill
   \begin{enumerate}[(i)]
       \item In general, the spectrum does not depend continuously on the metric in the compact-open topology on $M\setminus S$, even not when considering only manifolds with cylindrical ends of the same link. The advantage of introducing the $s$ is that now invertibility of the Dirac operator is an open property on $\Riemcu$, cp. Lemma~\ref{lem:spec_cyl}. 
 \item  Note that $(g,s)\in \Rc$ implies $(g,s')\in \Rc$ for all $s'\in (0,s)$.
 \end{enumerate}
  \end{remark}
 
 Let $\widetilde{S}\subset \widetilde{M}$ be the surgery sphere of the reverse surgery. This will be a surgery of codimension $k+1$. Then, $M\setminus S=M\setminus (D^{n-k}\times S^k) \sqcup_{S^{n-k-1} \times S^k} (D^{n-k}\setminus\{0\} \times S^k)$ and $ \widetilde{M}\setminus \widetilde{S}=M\setminus (D^{n-k}\times S^k) \sqcup_{S^{n-k-1} \times S^k} (S^{n-k-1}\times D^{k+1}\setminus\{0\} )$ are diffeomorphic by fixing a diffeomorphism $D^{n-k}\setminus\{0\} \times S^k\cong S^{n-k+1}\times (0,1) \times S^k \cong S^{n-k-1}\times D^{k+1}\setminus \{0\}$. Let $\widetilde{K}\subset \widetilde{M}$ be such that $\widetilde{K}\setminus \widetilde{S}\cong K\setminus S$ under the above diffeomorphism. We use on $\widetilde{K}\setminus \widetilde{S}$ the same coordinates as on $K\setminus S$ given by $(0,2]\times S^{n-k-1}\times S^k$ as in the beginning of this section.   Hence, $\Rc$ and $\Rcrev$ are homeomorphic.\medskip
 
Assume we can prove that $\RM\cong \Rc$ for codimension $n-k\geq 2$: Then, this statement applied  to the reverse surgery (codimension $k+1\geq 2$) gives Proposition~\ref{prop:main2} by  $\mathcal{R}^{\text{inv}}(\widetilde{M})\cong \Rcrev\cong \Rc\cong \RM$. 
 \tablecaption{Table of notations}\label{tab1}~
 \begin{supertabular}{|c|l|c|}
 \hline 
 Notation& Explanation & Ref. \\
 \hline
 \hline
 $\Riem$ & Riem. metrics on $M$ &\\
 $\RM$ & Riem. metrics on $M$ with invertible Dirac operator
 &p.~\pageref{page:Rinv}\\
 $\text{Riem}_{\text{cyl}}(M\setminus S)$ & An element $(g,s)$ is given by a Riemannian metric $g$ & p.~\pageref{page:standard_form}\\
 &  on $M\setminus S$ that has cylindrical end w.r.t. $\ln r$ for $r\leq s$&\\
& and standardizes structure for $r\leq 1$.& \\ $\Rc$ & $(g,s)\in \text{Riem}_{\text{cyl}}(M\setminus S)$ s.t. $g$ has invertible Dirac op. & p.~\pageref{page:inv_cyl}\\
 $\Rf$ & Metrics in $\RM$ with half-product structure near $S$ & \eqref{eq_Rf}\\
 \hline
 \end{supertabular}

\subsection{'Half-Flattening' and standardizing of metrics in \texorpdfstring{$\RM$}{RM} near \texorpdfstring{$S$}{S}}\label{step2}

The aim of this step is to show that the space of Riemannian metrics with invertible Dirac operator  is homotopy equivalent to a subspace of metrics which have product form with $SO(n-k)$-symmetry on a neighbourhood around the embedded surgery sphere $S$.\medskip 

Consider the set
\begin{align}\label{eq_US} U_{S,g}(\epsilon)\define\left\{p\in M\ |\ d_g(p,S)\leq \epsilon\right\}.\end{align}
Moreover, let $\exp_g^\perp\colon V_g\subset \mathbb R^{n-k}\times S\to M$ be the normal exponential map to $S$ w.r.t. $g$ which is well-defined on an open subset $V_g$ around $0\in \mathbb R^{n-k}$. For that we fix once and for all a trivialization of the normal bundle of $S$ in $TM$. We have $U_{S,g}(\epsilon)= \exp_g^\perp \circ (\exp_{g_0}^\perp)^{-1}(D_\epsilon)$ for $\epsilon\in (0,2]$.\medskip

The above goal will be obtained by a homotopy equivalence of $\RM$ to a space of half-flattened metrics, defined below, glued together from three steps:
\begin{enumerate}[(A)]
 \item\label{A}  The metric will be perturbed into a \emph{half-flat standard form around $S$}, i.e., there will be a continuous function $\epsilon\colon \RM\to (0,1)$ such that $g$ is homotopic to a metric $\hat{g}$ which on $U_{S,g}(\epsilon(g))$ has the form $\hat{g}|_{U_{S,g}(\epsilon(g))}= (\exp_{{g}}^\perp)_* (\xi_{n-k}+ {g}|_S)$, where $\xi_{n-k}$ is the euclidean metric on $\mathbb R^{n-k}$.  This will be obtained by a parametrized version of \cite[Lemma 3.4]{ADH}, see Section~\ref{sec:half}.
 \item\label{B} Using appropriate diffeomorphisms, $U_{S,g}(r)$ will be mapped onto $D_r$ for all $r\leq \epsilon(g)$. On the tubular neighbourhoods $U_{S,g}(\epsilon(g))$, this will be done via normal exponential maps. This will then be extended to all of $M$ via a parametrized version of the diffeotopy extension theorem, see Lemma~\ref{lem:stand}. \smallskip 
 
   \begin{minipage}{0.6\textwidth} Using further diffeomorphisms that are radial w.r.t.~$g_0$, we will finally obtain  metrics that have $SO(n-k)$-symmetry on all of $D_1$. For that we choose a continuous family of smooth monotonically increasing functions $\{ a_{\epsilon}\colon [0,2]\to [0,2]\}_{\epsilon\in (0,1]}$ with $a_{\epsilon}(r)=r$ for $r\in (0, \frac{\epsilon}{4})$, $a_{\epsilon}(1)=\epsilon$,  $a_\epsilon|_{[3/2,2]}=\mathrm{id}$ and $a_1\equiv \mathrm{id}$. With these functions we can define the subspace we were heading to:\end{minipage}
 \hfill
 \begin{minipage}{0.3\textwidth}
 \centering
  \begin{tikzpicture}[scale=0.56]
   \draw[->] (-0.5,0)--(5,0) node[below] {\small $r$};
   \draw[->] (0,-0.5)--(0,5) node[left] {\small $a_\epsilon$};
\draw[dotted] (0,0) -- (4,4);
\draw (0,0)--(0.6,0.6);
\draw (3,3)--(4,4);
\draw (-0.1,1.5) node[left] {\small $\epsilon$} -- (0.1, 1.5);
\draw[fill] (2,1.5) circle(0.03cm);
\draw (0.6,-0.1) node[below] {\small $\frac{\epsilon}{4}$} -- (0.6, 0.1);
 \draw (2,-0.1) node[below] {\small $1$} -- (2, 0.1);
\draw (4,-0.1) node[below] {\small $2$} -- (4, 0.1);
\draw (3,-0.1) node[below] {\small $\frac{3}{2}$} -- (3, 0.1);
 \draw (-0.1,0.6) node[left] {\small $\frac{\epsilon}{4}$} -- (0.1,0.6);
 \draw (-0.1,2) node[left] {\small $1$} -- (0.1,2);
 \draw (-0.1,4) node[left] {\small $2$} -- (0.1,4);

 \draw (-0.1,3) node[left] {\small $\frac{3}{2}$} -- (0.1,3);
 \draw (0.6,0.6) .. controls (1.4,1.4) and (1.8,1.5) .. (2,1.5) .. controls (2.2,1.5) and (2.6,2.6) .. (3,3);
 \end{tikzpicture} 
 \end{minipage}                                                                                                                                                                                                                                                                                                                   
\begin{align} \Rf \define\{ g&\in \RM\ |\  \exists \delta\in (0,1)\colon \nonumber\\ &g|_{D_1\setminus S} = a_{\delta}'(r)^2dr^2 + a_{\delta}(r)^2\sigma_{n-k-1} + g|_S\}.\label{eq_Rf}\end{align}
In particular, the $\delta$ in the definition of $\Rf$ is uniquely determined by the chosen $g\in \Rf$ since $a_\delta(1)=\delta$. Hence, we obtain an induced  function
$\delta\colon \Rf\to (0,1)$, which maps $g$ to its correponding $\delta$. This function is continuous. Note that by construction $D_1=U_{S,g}(\delta(g))$ for all $g\in \Rf$. Moreover, $\Rf$ is a closed subspace of $\RM$ as can be seen as follows: let $g_i\in \Rf \to g\in \RM$ with $\delta_i\define \delta(g_i)$. There cannot be a subsequence of $\delta_i$ that converges to $0$ since otherwise $g$ would no longer be a metric on $M$. Hence, $\delta_i$ convergence to some $\delta\in (0,1]$ and $g|_{D_1\setminus S}= a_\delta'(r)^2dr^2+a_\delta(r)^2\sigma_{n-k-1}+g|_S$.
 \end{enumerate}

In total we obtain 

\begin{proposition}[Proved on p.~\pageref{pr:Rf}]\label{prop:Rf}
The  space $\Rf$ is homotopy equivalent to $\RM$. 
\end{proposition}

\begin{remark}\label{rem:Rflat}
After \eqref{A} we  have already obtained a 'half-flat' metric and shown that $\RM$ is homotopy equivalent to 
\[ \Riemf \define \{g\in \RM\ |\ \exists \epsilon'\in (0,1)\colon g|_{U_{S,g}(\epsilon')}=(\exp_g^\perp)_* (\xi_{n-k}+g|_S)\}.\]
The drawback is that now there is no continuous function $\delta\colon\! \Riemf\to (0,1]$ such that $g|_{U_{S,g}(\delta(g))}=(\exp_g^\perp)_* (\xi_{n-k}+g|_S)$ for all $g\in \Riemf$ since along a continuous path in $\Riemf$ the $\delta$ can jump. But such a function will be needed in the next section to obtain a continuous blow-up into $\Rc$ (the $\Upsilon_\rho$ in the section below). Hence, \eqref{B} is mainly useful to further perturb the outcome of \eqref{A} to a space where we have such a continuous $\delta$---this space is $\Rf$. We note that \eqref{B} needs to be carried out careful enough such that the metrics in $\Rf$ stay in $\Rf$ throughout the homotopy.
\end{remark}

\subsection{\texorpdfstring{$\Rf\hookrightarrow \Rc$}{RwarpinRcyl}}
 
 Up to now we have established that $\RM$ is homotopy equivalent to $\Rf$. In order to get the desired homotopy equivalence to $\Rc$, we want to identify $\Rf$ with a subspace of $\Rc$. For that we give a continuous version of the blow-up map in \cite{ADH}:
 
 \begin{proposition}[Proved in Section~\ref{pr:existencerho}]
 \label{prop:existencerho}Let $\delta\colon \Rf\to (0,1)$ be as in \eqref{B} from above. Then there is a continuous function $\hat{\rho}\colon \Rf \to (0,1)$ with $\hat{\rho}\leq \delta/32$ such that for all continuous functions $\rho\colon \Rf\to (0,1)$ with $\rho\leq \hat{\rho}$ the map 
  \begin{align}\label{eq:Y}
   \Upsilon_\rho\colon \Rf&\to \Ri \times (0,1),\\
   \nonumber g &\mapsto \left(y_{\rho(g)}(g), \rho(g) \right),
     \end{align}  
where $y_{\rho(g)}(g)$ is defined as 
\[\left\{ \begin{matrix} \hspace{5cm} g\hspace{5.7cm}  \text{on }M\setminus D_1\\[0.2cm] F^2\left( a_{\delta(g)}'(r)^2dr^2 + a_{\delta(g)}(r)^2\sigma_{n-k-1} +  f_{\rho(g)}^2\left(\eta_{\rho(g)}g|_S + (1-\eta_{\rho(g)})\sigma_k\right)\right) \text{on } D_1\setminus S
                           \end{matrix}\right. \]
and
$F$, $f_{\rho}$ and $\eta_{\rho}$ are defined in \eqref{eq:defFfeta} (see also the left of Figure~\ref{fig:fcts}), is 
\begin{enumerate}[(i)]
\item a homeomorphism onto its image and
\item $\Upsilon_\rho (\Rf)$ is a closed subset of $\Rc$.
\end{enumerate}
\end{proposition}

Note that $y_{\rho(g)}(g)$ has a cylindrical end on $D_{\min\{\rho, \delta(g)/4\}}$ and, thus, it is complete. The $\rho$ will be chosen later.
\begin{figure}
\centering
\begin{tikzpicture}[scale=0.73]

\draw[->] (-0.5,0)--(5,0) node[below] {\small $r$};
\draw[->] (0,-0.5)--(0,5);

\draw (4.5,0.15)--(4.5,-0.15) node[below] {\small $1$};
\draw (3.5,0.15)--(3.5,-0.15) node[below] {\small $\frac{3}{4}$};
\draw (2.5,0.15)--(2.5,-0.15) node[below] {\small $\frac{1}{2}$};
\draw (2,0.15)--(2,-0.15) node[below] {\small $2\rho$};
\draw (1,0.15)--(1,-0.15) node[below] {\small $\rho$};

\draw (0.15,4.5)--(-0.15,4.5) node[left] {\small $2$};
\draw (0.15,3.5)--(-0.15,3.5) node[left] {\small $1$};
\draw (0.15,1)--(-0.15,1) node[left] {\small $\rho$};

\begin{scope}
\clip (0,-1) --(1,-1)--(1,2)--(0,2)--(0,-1);
\draw[color=green, line width=1] (0,0) node[xshift=17, yshift=25] {\small $r$}  -- (1,1) .. controls (1.3,1.2)  and (1.6,3.5) .. (2,3.5)--(5,3.5);
\draw[color=red, line width=1] (0,0) -- (1,0) .. controls (1.6,0.1)  and (1.4,3.45) .. (2,3.47)--(5,3.47);
\end{scope}

\begin{scope}
\clip (1,-1) --(2,-1)--(2,3.7)--(1,3.7)--(1,-1);
\draw[color=green, line width=1, dashed] (0,0) -- (1,1) .. controls (1.3,1.2)  and (1.6,3.5) .. (2,3.5)--(5,3.5);
\draw[color=red, line width=1, dashed] (0,0) -- (1,0) .. controls (1.6,0.1)  and (1.4,3.45) .. (2,3.47)--(5,3.47);
\end{scope}

\begin{scope}
\clip (2,-1) --(5,-1)--(5,3.54)--(2,3.54)--(2,-1);
\draw[color=green, line width=1] (0,0) -- (1,1) .. controls (1.3,1.2)  and (1.6,3.5) .. (2,3.5)--(5,3.5);
\draw[color=red, line width=1] (0,0) -- (1,0) .. controls (1.6,0.1)  and (1.4,3.45) .. (2,3.47)--(5,3.47);
\draw[color=blue, line width=1] (0.5,5) .. controls (1, 4.6) .. (2.5,4.5) .. controls (2.9,4.4)  and (3.1,3.52) .. (3.5,3.53)--(5,3.53);
\end{scope}

\begin{scope}
\clip (0,6) --(2.5,5.5)--(2.5,4.4)--(0,4.4)--(0,4.4);
\draw[color=blue, line width=1] (0.5,5) node[xshift=30, yshift=-3] {\small $\frac{1}{r}$} .. controls (1, 4.6) .. (2.5,4.5) .. controls (2.9,4.4)  and (3.1,3.52) .. (3.5,3.53)--(5,3.53);
\end{scope}

\begin{scope}
\clip (2.4,4.5)--(3.5,4.4)--(3.5,3.5)--(2.5,3.5) --(2.5,4.4);
\draw[color=blue, line width=1, dashed] (0.5,5) .. controls (1, 4.6) .. (2.5,4.5) .. controls (2.9,4.4)  and (3.1,3.52) .. (3.5,3.53)--(5,3.53);
\end{scope}

\begin{scope}[shift={(8,3)}, scale=0.5] 

\draw (-1.5,-0.5) node {\small $(M,g)$};
\draw[line width=1] (0,0) .. controls (-2,1) and (-2,2) .. (0,3) .. controls (2,4) and (3,4) .. (4,3.7) .. controls (5, 3.6) and (5.5, 4) .. (6, 4) .. controls (12, 4.5) and (12, -1.5) .. (6,-1) .. controls (5.5,-1) and (5,-0.6) .. (4, -0.7) .. controls (3,-1) and (2,-1) .. (0,0);
\draw[line width=1] (0,2) .. controls (1,1.5) .. (2,2);
\draw[line width=1] (0.4,1.8) .. controls (1,2) .. (1.6,1.8);
\begin{scope}[shift={(2,-1)}]
\draw[line width=1] (0,2) .. controls (1,1.5) .. (2,2);
\draw[line width=1] (0.4,1.8) .. controls (1,2) .. (1.6,1.8);
\end{scope}

\draw[dashed] (7,5.2) node[above] {\tiny $r=1$ } -- (7,-7.5);
\draw[dashed] (8.6,4.5) node[above, yshift=-3] {\tiny  $r=\frac{1}{2}$ } -- (8.6,-7.5);
\draw[dashed] (10,5.2) node[above] {\tiny  $r=\rho$ } -- (10,0) ..controls (10,-2) and (11,-2) .. (11,-3.7)  --(11,-7.5);
\draw[dashed] (10.5,4.5) node[above, xshift=6] {\tiny  $r=0$ } -- (10.5,0) .. controls (10.5,-2) and (13,-3.5) .. (14,-3.5);
\draw[fill] (10.5,1.5) circle (0.05cm);
\draw[->] (11.5,1.5) node[right] {\small  $S$}-- (10.7,1.5);

\begin{scope} [shift={(0,-6)}]
\draw (-1.5,-1) node {\small $(M\setminus S, \Upsilon_\rho(g))$};

\draw[line width=1] (0,0) .. controls (-2,1) and (-2,2) .. (0,3) .. controls (2,4) and (3,4) .. (4,3.7) .. controls (5, 3.6) and (5.5, 4) .. (6, 4)  ..     controls (9.2,4.3) and (9.4,3.1) .. (9.5,3) .. controls (10.5, 1.5) and (10.8, 2) .. (11,2) -- (14,2); 
\draw[line width=1] (14,1) -- (11,1) .. controls (10.8,1) and (10.5, 1.5) .. (9.5,0) .. controls (9.4, 0) and (9.2, -1.3) .. (6,-1) .. controls (5.5,-1) and (5,-0.6) .. (4, -0.7) .. controls (3,-1) and (2,-1) .. (0,0);
\draw[line width=1] (0,2) .. controls (1,1.5) .. (2,2);
\draw[line width=1] (0.4,1.8) .. controls (1,2) .. (1.6,1.8);
\begin{scope}[shift={(2,-1)}]
\draw[line width=1] (0,2) .. controls (1,1.5) .. (2,2);
\draw[line width=1] (0.4,1.8) .. controls (1,2) .. (1.6,1.8);
\end{scope}
\end{scope}

\end{scope}
\end{tikzpicture}

\caption{Left: $F(r)$ (blue), $f_{\rho}(r)$ (green) and $\eta_\rho(r)$ (red).\newline  Right: How $\Upsilon_\rho$ changes the metric. There are no changes for $r\geq 3/4$. For $r\leq \rho(g)$ the resulting metric has the cylindrical end $r^{-2}dr^2+\sigma_{n-k-1}+\sigma_k$. The radial scale on the upper and lower picture is different; a comparison is given by the dashed lines.}
\label{fig:fcts}
\end{figure}

 \subsection{\texorpdfstring{$\Upsilon_\rho(\Rf)\cong \Rc$}{RwarpRcyl} for \texorpdfstring{$\rho\colon \Rf\to (0,1)$}{rho} small enough.}
 
 In this step we will see that for $\rho$ small enough $\Upsilon_\rho(\Rf)$ and $\Rc$ are homotopy equivalent which together with the previous steps will finish the proof.

\subsubsection{Torpedos}\label{sec:torp}

 We choose a Riemannian metric $h_{m}$ on $\mathbb R^{m}$, $m\geq 2$, such that in standard spherical coordinates  $h_m =d\tilde{r}^2+w(\tilde{r})^2\sigma_{m-1}$ with radial coordinate $\tilde{r}$, $w(\tilde{r})=\tilde{r}$ for $\tilde{r}\leq 1$  and $w(\tilde{r})=1$ for $\tilde{r}\geq 2$.  We choose $h_m$ such that $D^{h_m}$ is invertible, see Appendix~\ref{app}.\medskip
 
 The \emph{$(n,k)$-torpedo}  $(T_{n,k}, h_{torp}^{n,k})$ is then defined to be the product manifold $(T_{n,k}=\mathbb R^{n-k} \times S^k, h_{torp}^{n,k}=h_{n-k}+\sigma_k)$.\medskip

 \subsubsection{Grafting of torpedos on metrics in $\Rc$}
 
Firstly, we will see that for any $(g,s)\in \Rc$ and a gluing cylinder of large enough length $L$ the glued manifold
\begin{align*}
 M\setminus \{r<s\} \cup_{\{r=s\}\subset M} [0,L] \times S^{n-k-1}\times S^k \cup_{\{\tilde{r}=2\}\subset T_{n,k}} T_{n,k}\setminus \{\tilde{r}\geq 2\},
\end{align*}
with metric induced from $g$ and $h_{\text{torp}}^{n,k}$, is an element in $\RM$, cp. Figure~\ref{fig:glue_cyl}. In particular, after using a radial diffeomorphism the resulting metric $\hat{g}$ should be an element in $\Rf$ (with $\delta (\hat{g})=1$). Then, $\Upsilon_\rho(\hat{g})$ shall be the metric to which we want to deform $(g,s)$. We need to choose $L$ depending continuously on $(g,s)$ and such that an appropriate interpolation of $(g,s)$ and $\Upsilon_\rho(\hat{g})$ will be in $\Rc$ as well.\medskip   

For that we need a parametrized version of the gluing result for cylindrical manifolds from \cite[Prop. 2.1]{Dahl.2008}: 

\begin{definition}\label{def:cyl} Let $N$ be a manifold with (at least) one end  $Z_N\subset N$ such that $Z_N$ is diffeomorphic to $[2,\infty)\times S^{n-k-1}\times S^k$.  For $c>0$, let $\text{R}_{c}(N)$ be  the set of complete Riemannian metrics on $N$   such that under the above diffeomorphism $h|_{Z_N}=d\hat{r}^2+\sigma_{n-k-1}+\sigma_k$ and such that $\text{inf}\, \text{spec} (D^h)^2\geq c^2$.
\end{definition}

Note that we do neither assume that $N\setminus Z_N$ is compact nor that all ends are cylindrical. Moreover, we can view $(T_{n,k}, h_{torp}^{n,k})$ as an element of $R_{k^2/4}(T_{n,k})$ and $\hat{r}=\tilde{r}$.

\begin{lemma}[Proved on p.~\pageref{pr:L}]\label{lem:L}  Let $N$ be a manifold as in Definition~\ref{def:cyl}. Let $c\colon \Rc\to (0,\infty)$ be continuous. Then there is a continuous map $L\colon \Rc \to (0,\infty)$ such that for all $(g,s)\in \Rc$, $h\in \text{R}_{c(g,s)}(N)$ and $L\geq L(g,s)$ the manifold
\[ Z\define M\setminus \{r<s\} \cup_{\{r=s\}\subset M} [0,L] \times S^{n-k-1}\times S^k \cup_{\{\hat{r}=2\}\subset N} N\setminus \{\hat{r}\geq 2\}\]
with the glued metric called $Z((g,s),h, L)$, see Figure~\ref{fig:glue_cyl}, has invertible Dirac operator. Moreover, $L$ can be chosen such that it depends continuously on $c$.
\end{lemma}

\begin{figure}
\centering
\begin{tikzpicture}[scale=0.71] 

\draw (2,-1.4) node { $(M\setminus \{r \leq s\} ,g)$};
\draw (12.5,-1.4) node { $(N\setminus \{\hat{r} \geq 2\} ,h)$};

\draw[line width=1] (0,0) .. controls (-2,1) and (-2,2) .. (0,3) .. controls (2,4) and (3,4) .. (4,3.7) .. controls (5, 3.6) and (5.5, 4) .. (6, 2)  ..     controls  (6.5, 1.5) and (6.8, 2) .. (7,2) -- (10,2) .. controls (10.2,1.9) and (10.4, 2.5) .. (11,3) .. controls (11.2,3.2) and (11.4, 3.4) .. (13, 3); 
\draw[line width=1] (13,0) .. controls (11.4, -0.4) and (11.2, -0.2) ..(11,0).. controls (10.4,0.5 ) and (10.2, 1.1) .. (10,1) -- (7,1) .. controls (6.8, 1) and (6.5, 1.5) .. (6,1) .. controls (5.5,-1) and (5,-0.6) .. (4, -0.7) .. controls (3,-1) and (2,-1) .. (0,0);
\draw[line width=1] (0,2) .. controls (1,1.5) .. (2,2);
\draw[line width=1] (0.4,1.8) .. controls (1,2) .. (1.6,1.8);
\begin{scope}[shift={(2,-1)}]
\draw[line width=1] (0,2) .. controls (1,1.5) .. (2,2);
\draw[line width=1] (0.4,1.8) .. controls (1,2) .. (1.6,1.8);
\end{scope}

\draw [decorate,decoration={brace,amplitude=10pt},xshift=-4pt,yshift=0pt]
 (9.7,0.5) -- (7.3,0.5) node [midway,yshift=-17pt] {\footnotesize length
$L(g,s)$};

\draw[dashed] (7,3) -- (7,-1) node[below] {\small $r=s$};
\draw[dashed] (9.7,3) -- (9.7,-1) node[below] {\small $\hat{r}=2$};

\end{tikzpicture}
\caption{$(M\setminus S,g)$ with $(g,s)\in \Rc$ is glued to $(N,h)$ with a cylindrical end on $\{\hat{r}\geq 2\}$ via a cylinder of length $L(g,s)$. If the Dirac operator on $(N,h)$ has a spectral gap, then $L(g,s)$ can be chosen large enough that the resulting metric has again an invertible Dirac operator, see Lemma~\ref{lem:L}.} \label{fig:glue_cyl}
\end{figure}

With this preparation we will put the idea from above into a proposition:

\begin{proposition}[Proved in Lemma~\ref{lem:cup}]\label{prop:cup}
 For  $\rho\define \hat{\rho} \colon \Rf\to (0,1)$ from  Proposition~\ref{prop:existencerho} there is a map
 \begin{align*}
  \Xi_{\mathrm{gr}}\colon& \Rc \times [0,1] \to \Rc
 \end{align*}
that  is 
 \begin{enumerate}[(i)]
  \item well-defined and continuous,
\item $\Xi_{\mathrm{gr}}(.,0)=\mathrm{id}$,
\item  $\Xi_{\mathrm{gr}}((g,s),1) \in  \Upsilon_\rho(\Rf)$ for all $(g,s)\in \Rc$.
\end{enumerate}
\end{proposition}

Note that this map will \underline{not} map $\Upsilon_\rho(\Rf)$ into itself for all $t\in [0,1]$. Hence, it is not yet our desired homotopy equivalence $\Upsilon_\rho(\Rf)\cong \Rc$. At the end, we will use $\Xi_{\mathrm{gr}}$ away from $\Upsilon_\rho(\Rf)$. Near this subset we will use a different map that will be specified below.

 \subsubsection{Near \texorpdfstring{$\Rf$}{Rf}}\label{step6_1}

 In this section we write down a deformation retract from an open neighbourhood of $\Upsilon_\rho(\Rf)$ to $\Upsilon_\rho(\Rf)$.\medskip 
 
 For that we construct an extension of $\Upsilon_\rho\colon \Rf\to \Upsilon_\rho(\Rf)$ and of its inverse to $\RM$ and $\Rc$, respectively.\medskip 
 
 Let us first collect some auxiliary functions: Let $\rho \colon \Rf\to (0,1)$ be the function of Proposition~\ref{prop:cup} and $\delta\colon \RM\to (0,1)$ as chosen in \eqref{def_delta}.  We choose a continuous extension of $\rho$ to $\RM$ with image in $(0,1)$, also called $\rho$ in the following, which is possible by Tietze's extension theorem since $\Rf$ is closed in $\RM$ (cp. \eqref{def_delta} and above).\medskip 

We then extends the map $\Upsilon_\rho\colon \Rf\to \Rc$ to a continuous map  $\hat{\Upsilon}_\rho\colon \RM\to \Riemcu\subset \Ri\times (0,1)$, $g\mapsto (\hat{y}_{\rho(g)}(g), \rho(g))$, as follows: Let $\nu\colon M\to [0,1]$ be a smooth function such that $\nu\equiv 1$ on $M\setminus D_{1}$ and $\nu\equiv 0$ on $D_{3/4}$. We set 
\begin{align*}
 \hat{y}_{\rho(g)}&(g)\, \define  \nu g  \\
 +&(1-\nu) F^2 \left( a_{\delta(g)}'(r)^2dr^2+a_{\delta(g)}(r)^2\sigma_{n-k-1} + f_{\rho(g)}^2(\eta_{\rho(g)} g|_S + (1-\eta_{\rho(g)})\sigma_k)\right).
\end{align*}
 By construction, $\hat{y}_{\rho(g)}$ has a cylindrical end for $r\leq \rho(g)$ and, hence, $\hat{\Upsilon}_\rho$ really maps into $\Riemcu$. Additionally, $\hat{\Upsilon}_\rho=\Upsilon_\rho$ on $\Rf\subset \RM$ since $g=y_{\rho(g)}(g)$ on $M\setminus D_{3/4}$ for all $g\in \Rf$.\medskip

Moreover, the homeomorphism, cp. Proposition~\ref{prop:existencerho}.(i), \[\Upsilon_\rho\colon \Rf\to \Upsilon_\rho(\Rf)\subset \Rc\] gives rise to the $\check{\rho}\define \rho\circ  \Upsilon_\rho^{-1}$ from $\Upsilon_\rho(\Rf)$ to $(0,1/32)$. Since by Proposition~\ref{prop:existencerho}(ii) $\Upsilon_\rho(\Rf)$ is closed in $\Rc$, we can also extend 
this map to some continuous $\check{\rho}\colon \Rc\to (0,1/32)$.\medskip

We can now also extend $\Upsilon_\rho^{-1}$ to a map $A\colon \Rc\to \Riem$: We choose a smooth function $\tau\colon (0,1]\times (0,1]\to (0,1]$ with \[ \tau_a(r)\define \tau(a,r) \define \left\{\begin{matrix}
         1 & r\leq 2a\\
         0 & r\geq 3a                                                                                                                                                                                                                   \end{matrix}
\right.\]
and define $A(g,s)$ to be 
\begin{align*}
 \left\{ 
 \begin{matrix}
\hspace{4.5cm}g \hspace{5.5cm} \text{on }M\setminus D_1  \\
 F^{-2}\left( z(r)^2dr^2 + g^{n-k-1}(r) + f_{\check{\rho}(g,s)}^{-2} \left((1-\tau_{\check{\rho}(g,s)})g^k(r)\right)\right)+ \tau_{\check{\rho}(g,s)} g^k(1)\  \text{on }D_1
 \end{matrix}
 \right.
\end{align*}
where $z(r)^2dr^2+g^{n-k-1}(r)+g^k(r)=g|_{D_1}$, cp. (II) in Section~\ref{step1}.\medskip 

We note that $g=\frac{dr^2}{r^2}+\sigma_{n-k-1}+\sigma_k$ for $r\leq s$ and $F, f_\rho$ are $1$ on $r\in (3/4,1)$. Together with $A(g,s)= dr^2+r^2\sigma_{n-k-1}+g^k(1)$ on $r\leq \min \{s, \check{\rho}(g,s)\}<1/32$ this  implies that $A(g,s)$ really gives a metric on $M$ and that $A$ is continuous.\medskip 
 
  For $(g,s)= \Upsilon_\rho(\hat{g})$ it is $\check{\rho}(g,s)=\rho(\hat{g})$, $z=F a_{\delta(\hat{g})}'$, $g^{n-k-1}=F^2a_{\delta(\hat{g})}^2\sigma_{n-k-1}$ and $g^k=F^2f_{\rho(\hat{g})}^2(\eta_{\rho(\hat{g})}\hat{g}|_S+(1-\eta_{\rho(\hat{g})})\sigma_k)$. In particular,  we have $g^{k}(1)=\hat{g}|_S$ and, hence, 
  \[ (1-\tau_{\rho(\hat{g})})(\eta_{\rho(\hat{g})}\hat{g}|_S+(1-\eta_{\rho(\hat{g})})\sigma_k)+\tau_{\rho(\hat{g})} g^k(1)=\hat{g}|_S\quad \text{on }r\leq 1.
  \]
 This implies $A(g,s)=\hat{g}$. Hence, $A$ really extends $\Upsilon_\rho^{-1}$.\medskip

Up to now we obtained continuous extensions $\hat{\Upsilon}_\rho\colon \RM\to \Riemcu$ and $A\colon \Rc\to \Riem$ of $\Upsilon_\rho$ and $\Upsilon_\rho^{-1}$, respectively. 
 Let further $T_\epsilon\colon \RM\times [0,1] \to \RM$ be the homotopy \eqref{map:T} for $\RM\cong \Rf$ from the proof of Proposition~\ref{prop:Rf}. This means in particular $T_\epsilon (\Rf \times [0,1])\subset \Rf$ and $T_\epsilon (\RM\times \{1\})\subset \Rf$.\medskip 
 
Since $\Upsilon_\rho(\Rf)\subset \Rc$ and invertibility is an open property on $\Riemcu$ by Remark~\ref{rem:R_metric}, there is an open neighbourhood $\tilde{\mathcal{D}}$ of  $\Upsilon_\rho(\Rf)$ in $\Rc$ with $\tilde{\mathcal{D}}\subset \Rc$. Then $(\Upsilon_\rho\circ T_\epsilon)^{-1}(\tilde{\mathcal{D}})$ is an open neighbourhood of $\Rf\times [0,1]$ in $\RM\times [0,1]$. Hence, there exists an open neighbourhood $U$ of $\Rf$ in $\RM$ with $U\times [0,1]\subset (\hat{\Upsilon}_\rho\circ T_\epsilon)^{-1}(\tilde{\mathcal{D}})$. Since invertibility is an open property also in $\Riem$, $U$ is also an open neighbourhood of $\Rf$ in $\Riem$. Thus, $\mathcal{D}_1\define  A^{-1}(U)$ an open neighbourhood of $\Upsilon_\rho(\Rf)$ in $\Rc$ such that  $\hat{\Upsilon}_\rho(T_\epsilon (A(\mathcal{D}_1)\times [0,1]))\subset \Rc$.\medskip

Again since invertibility is an open property in $\Riemcu$, for every $h\in \Rf$ there is an $\epsilon(h)>0$ such that $B_{\epsilon(h)}(\Upsilon_\rho(h))\subset \Rc$. Moreover, since $\hat{\Upsilon}_\rho \circ A$ is continuous and on $\Upsilon_\rho(\Rf)$ the identity, for each $h\in \Rf$ there is a $\delta(h)\in (0, \epsilon(h))$ such that 
for all $(g,s)\in \Rc$ with $\Vert (g,s)-\Upsilon_\rho(h)\Vert<\delta(h)$ it is 
\[\Vert \hat{\Upsilon}_\rho\circ A((g,s))-\Upsilon_\rho(h)\Vert=\Vert \hat{\Upsilon}_\rho\circ A((g,s))-\hat{\Upsilon}_\rho\circ A(\Upsilon_\rho(h))\Vert<\epsilon(h).\]

Set $\mathcal{D}_2=\cup_{h\in \Rf} B_{\delta(h)}(\Upsilon_\rho(h))\subset \Riemcu$. By the choice of $\delta(h)$ this is an open neighbourhood of $\Upsilon_\rho(\Rf)$ in $\Rc$. Moreover, for all $(g,s)\in \mathcal{D}_2$ we have $a_t\define (1-2t)(g,s)+ 2t \hat{\Upsilon}_\rho(A(g,s))\in \Rc$ for $t\in [0,1]$ as can be seen as follows: By construction of $\mathcal{D}_2$ there is an $h\in \Rf$ with $\Vert (g,s) - \Upsilon_\rho(h)\Vert<\delta(h)$. We estimate

 \begin{align*}
  \Vert a_t- \Upsilon_\rho(h)\Vert\leq& (1-t)\Vert (g,s)- \Upsilon_\rho(h)\Vert + t\Vert \hat{\Upsilon}_\rho(A(g,s))-  \Upsilon_\rho(h)\Vert\\ < &(1-t)\delta(h)+t\epsilon(h)< \epsilon(h).
 \end{align*}
Thus, $a_t\in \Rc$ by the choice of $\epsilon(h)$.\medskip

We set $\mathcal{D}\define \mathcal{D}_1\cap \mathcal{D}_2$ which is an open neighbourhood of $\Upsilon_\rho(\Rf)$ in $\Rc$. Altogether we obtain:

\begin{lemma} \label{prop:nearRf}  The map 
\begin{align*}
 \Xi_{\text{near}}\colon  \mathcal{D} \times [0,1] &\to \Rc,\\
 \nonumber  ((g,s),t) &\mapsto \left\{ \begin{matrix} (1-2t)(g,s)+ 2t \hat{\Upsilon}_\rho(A(g,s)) &\text{for }t\in [0,1/2],\\
 \hat{\Upsilon}_\rho\left(T_\epsilon(A(g,s),(2t-1))\right)&\text{for }t\in (1/2,1]\end{matrix}\right.
 \end{align*}
 is well-defined and continuous and fulfils 
\begin{align*} 
\Xi_{\text{near}}((g,s),0)=&(g,s)  &&\text{for all }(g,s)\in \mathcal{D}\\
\Xi_{\text{near}}((g,s),t)\in &\Upsilon_\rho(\Rf) &&\text{for all }(g,s)\in \Upsilon_\rho(\Rf),\ t\in [0,1], \\
 \Xi_{\text{near}}((g,s),1)\in &\Upsilon_\rho(\Rf) &&\text{for all }(g,s)\in \mathcal{D}.
 \end{align*}

\end{lemma}
 
\begin{proof}  $\Xi_{\text{near}}(.,0)=\mathrm{id}$ is clear by definition. Moreover, \[\Xi_{\text{near}}((g,s),1)=\hat{\Upsilon}_\rho\left(T_\epsilon(A(g,s),1)\right)\in \hat{\Upsilon}_\rho (\Rf)= {\Upsilon}_\rho(\Rf).\]
Continuity of $\Xi_{\text{near}}$ follows from $T_\epsilon(.,0)=\mathrm{id}$ and the continuity of the involved maps.\medskip 

For $t\leq 1/2$ the image is in $\Rc$ since $\mathcal{D}\subset \mathcal{D}_2$, see above. For $t>1/2$ the image is in $\Rc$ since $A(\mathcal{D})\subset A(\mathcal{D}_1)$ and $\hat{\Upsilon}_\rho(T_\epsilon (A(\mathcal{D}_1)\times [0,1]))\subset \Rc$, see above.\medskip

It remains to check the properties of this map on $\Upsilon(\Rf)$: For $(g,s)\in \Upsilon_\rho(\Rf)$ it is $\hat{\Upsilon}_\rho(A(g,s))=(g,s)$ and, thus, for $t\leq 1/2$ the map is the identity. For $t>1/2$ and $(g,s)\in \Upsilon_\rho(\Rf)$ it follows by $T_\epsilon (\Rf \times [0,1])\subset \Rf$  that $\Xi_{\text{near}}((g,s),t)\in \Upsilon_\rho(\Rf)$.
\end{proof}
\subsubsection{Gluing together}

For our homotopy from $\Rc$ to $\Upsilon_\rho(\Rf)$
we use $\Xi_{\text{near}}$ for elements near enough to $\Upsilon_\rho(\Rf)$.  All other elements in $\Rc$ will first be  moved by $\Xi_{\mathrm{gr}}$ into this neighbourhood, cp. Figure~\ref{fig:gluetog}.\medskip 

\begin{figure}
\centering

\begin{tikzpicture}[scale=0.97]
\draw[line width=1] (4,0) .. controls (6,-1) and (8,1) .. (10,0);
\draw (7,3) .. controls (9,2) and (11,4) .. (13,3);
\draw (4,0) .. controls (5, 1) and (6,0) .. (7,3);
\draw (10,0) .. controls (11, 1) and (12,0) .. (13,3);

\draw[color=darkyellow, line width=1]  (6,0.5) .. controls (7.09,1.604)   .. (8.64,2.03) .. controls (7.8,1.6)  .. (9.32,1.5);
\draw[color=darkyellow] (7.6,2) node {\small $\Xi(z,.)$};

\draw[color =darkblue] (7.6,-0.4) node {\small $\Upsilon_\rho(\Rf)$};
\draw[line width=1, color =blue] (7.5, 0.1) .. controls (9,1) .. (10.5, 3.1);

\draw (11.4,0.2) node {\small $\Rc$};

\draw[color=darkred, dashed, line width=1] (6,0.5) .. controls (7,1.6) .. (10,2.4);
\draw[color=darkred, fill] (6,0.55) node[below] {\small $z:=(g,s)$} circle (0.04cm);
\draw[color=darkred, fill] (10,2.4) circle (0.04cm);
\draw[color=darkred, fill] (8.64,2.03) circle (0.04cm);
\draw[color=darkgreen, fill] (9.32,1.5) node[right] {\small $\Xi_{\text{near}}(y,1)$} circle (0.04cm);
\draw[color=darkgreen, dashed, line width=1] (8.66,2) .. controls (7.8,1.6) .. (9.32,1.5);

\draw[color=darkred] (10.7,2.2) node {\small $\Xi_{\mathrm{gr}}(z,1)$};
\draw[color=darkred] (8.4,3.2) node {\small $y:=\Xi_{\mathrm{gr}}(z,t_{min})\in \partial \mathcal{U}$};

\draw[color=darkred, ultra thin, ->] (8.4,3) .. controls (7,2.5)  and (9,2.5) .. (8.7, 2.1);
 \begin{scope}
 \clip (4,0) .. controls (6,-1) and (8,1) .. (10,0)  .. controls (11, 1) and (12,0) .. (13,3) .. controls (11,4) and (9,2) .. (7,3) .. controls (6,0) and (5,1) ..  (4,0);
 \draw[fill, opacity=0.15] (7.2, -0.6) .. controls (8.3, 2) .. (11.3,4)  .. controls (10.5, 2) .. (7.2,-0.6);
\draw (11.1,2.9) node {\small $\mathcal{U}$};
 \end{scope}

\end{tikzpicture}
\caption{The construction of $\Xi$ using $\Xi_{\mathrm{gr}}$ outside the neighbourhood $\mathcal{U}\subset \Rc$ of $\Upsilon_\rho(\Rf)$ and $\Xi_{\text{near}}$ in $\mathcal{U}$ in Proposition~\ref{prop:sigma}.}\label{fig:gluetog}
\end{figure}

For that  let $\mathcal{U}$ be an open neighbourhood of $\Upsilon_\rho(\Rf)$ such that $\overline{\mathcal{U}}\subset \mathcal{D}$ for the $\mathcal{D}$ of Lemma~\ref{prop:nearRf}.  Then, the map
 \begin{align*}
  t_{\text{min}}\colon \Rc\to [0,1],\qquad 
  (g,s) \mapsto \text{inf} \{t\in [0,1]\ |\ \Xi_{\mathrm{gr}}((g,s),t) \in \overline{\mathcal{U}}\}
 \end{align*}
is continuous. Note that by Proposition~\ref{prop:cup}(iii) and the continuity of $\Xi_{\mathrm{gr}}$, $t_{\text{min}}<1$. Moreover, for $(g,s)\in \Upsilon_\rho(\Rf)$ it is $t_{\text{min}}=0$.\medskip 

From Lemma~\ref{prop:nearRf} and Proposition~\ref{prop:cup} we then directly obtain, cp. also Figure~\ref{fig:gluetog}: 

\begin{proposition}\label{prop:sigma}
The map 
\begin{align*} \Xi &\colon \Rc\times [0,1] \to \Rc\\
((g,s),t) &\mapsto \left\{ 
\begin{matrix}
 \Xi_{\mathrm{gr}} ((g,s),t) & t\leq t_{\text{min}}(g,s)\\
 \Xi_{\text{near}}\left(\Xi_{\mathrm{gr}} \left((g,s),t_{\text{min}}(g,s)\right), \frac{t-t_{\text{min}}(g,s)}{1-t_{\text{min}}(g,s)}\right) & t> t_{\text{min}}(g,s)
\end{matrix}
\right.
\end{align*}

is continuous with the following properties
\begin{enumerate}[(i)]
 \item $\Xi (. ,0)= \mathrm{id}$,
 \item $\Xi (. ,1)\subset \Upsilon_\rho (\Rf)$,
 \item $\Xi ((g,s),t)\subset \Upsilon_\rho(\Rf)$ for all $((g,s),t)\in \Upsilon_\rho (\Rf)\times [0,1]$.
\end{enumerate}
 In particular, $\Upsilon_\rho(\Rf)$ is homotopy equivalent to $\Rc$.
\end{proposition}

This finishes the proof of Proposition~\ref{prop:main2}. The very coarse structure of this proof is again summarized in Figure~\ref{fig:not}.

\begin{figure}
\newcommand{\gnoc}{\rotatebox[origin=c]{-90}{$\cong$}}
\begin{empheq}[box=\fbox]{align*}
\RM\!\!\!\!\!\!\! \underset{\text{Sec.}~\ref{step2}}{\overset{\text{homot. equiv.}}{\cong}} \!\!\!\Rf\!\underset{\text{Prop.}~\ref{prop:existencerho}}{\overset{\text{homeom.}}{\cong}}\!\!\!\Upsilon_\rho(\Rf)\!\!\!\!\!\underset{\text{Prop.}~\ref{prop:sigma}}{\overset{\text{homot. equiv.}}{\subset}}\!\!\!\!\!\Rc\\
\text{\tiny{homeom./Sec.~\ref{step1}}}\ \gnoc \hspace{0.7cm}\\
\mathcal{R}^{\text{inv}}(\widetilde{M})\!\!\!\!\! \underset{\phantom{\text{Sec.}~\ref{step2}}}{\overset{\phantom{\text{homot. equiv.}}}{\cong}}\!\!\!\!\!\mathcal{R}^{\text{inv}}_{\frac{1}{2}\text{flat}-D_1}(\widetilde{M})\!\!\!\!\!\underset{\phantom{\text{Prop.}~\ref{prop:existencerho}}}{\overset{\phantom{\text{homeom.}}}{\cong}}\!\!\!\Upsilon_\rho(\mathcal{R}^{\text{inv}}_{\frac{1}{2}\text{flat}-D_1}(\widetilde{M}))\!\!\!\underset{\phantom{\text{Prop.}~\ref{prop:sigma}}}{\overset{\phantom{\text{homot. equiv.}}}{\subset}}\!\!\!\!\!\Rcrev
\end{empheq}
\caption{Coarse structure of the proof of Proposition~\ref{prop:main2}.}\label{fig:not}
\end{figure}

\subsection{The relative version}\label{sec:rel}

Let $M^n$ and $\widetilde M^n$ be two  closed connected nonempty spin manifolds of dimension  $n\geq 3$ where $\widetilde{M}$ can be obtained from $M$ by a spin surgery of codimension $2\leq n-k\leq n-1$.  Let $A$ be a compact subset such that $M\setminus A$ contains the surgery sphere. If we choose $g_0$ in Subsection~\ref{step1} such that $D_2\subset M\setminus A$, then none of the following steps in the proof changes the metric on $A$.
Moreover, we can view $A$ as a compact subset of $\widetilde{M}$ as well.\medskip 

Let $\pi_A\colon \Riem\to \text{Riem}(A)$  and $\widetilde{\pi}_A\colon \text{Riem}(\widetilde{M})\to \text{Riem}(A)$ be both given by $g\mapsto g|_A$. Let $B\define \pi_A(\RM))$. By \cite[Thm.~1.1]{ADH11} $B=\widetilde{\pi}_A(\mathcal{R}^{\text{inv}}(\widetilde{M}))$. 
Restricting  $\pi_A $ resp. $\widetilde{\pi}_A$ to $\RM$ resp. $\mathcal{R}^{\text{inv}}(\widetilde{M})$ gives rise to maps $\RM\to B$ and $\mathcal{R}^{\text{inv}}(\widetilde{M})\to B$.\medskip 

Since none of the  steps in the proof of Proposition~\ref{prop:main2} actually changes the metric on $A$, we actually obtain:
  \begin{proposition}\label{prop:main3} With the notations from above, 
  $\RM$ and $\mathcal{R}^{\text{inv}}(\widetilde{M})$ are  homotopy equivalent over $B$.
\end{proposition}

\section{Half-flattening and standardizing}\label{pr:Riemf}

\subsection{Half-flattening}\label{sec:half}

As announced in Step~\ref{step2}\eqref{A} we want to give a parametric version of \cite[Prop. 3.2]{ADH}. \medskip

 Let $\mathrm{inj}^\perp\colon \Riem\to \mathbb R_+$ be such that $\mathrm{inj}^\perp(g)$ is the normal injectivity radius of $S$ w.r.t. $g$, i.e., the supremum of all   $\ell \in \mathbb R_+$ such that $\exp_g^\perp$ restricted to $D^\ell\times S$ is a diffeomorphism onto its image. The map $\mathrm{inj}^\perp$ is continuous, see \cite[p.~177]{Ehr} for the proof for $S=\{pt\}$, the proof for an embedded closed submanifold $S$ is analogous.\medskip
 
Let $\eta\colon \mathrm{dom}(\eta)\define\{ (\epsilon, p, g)\ |\ 2\epsilon <  \mathrm{inj}^\perp (g)\} \subset (0,1)\times M \times \Riem\to [0,1]$ be a continuous function such that  
$$\eta_{\epsilon, g}(p)\define \eta(\epsilon, p,g)\define\begin{cases}
                        1\quad p\in U_{S,g}(\epsilon)\\
                    0\quad p\in M\setminus U_{S.g}(2\epsilon),
        \end{cases}
$$
$|d\eta_{\epsilon,g}|_g                                                                                                                                  \leq \frac{2}{\epsilon}$,
and $\eta_{\epsilon, g}\colon M\to [0,1]$ is smooth for all $\epsilon$ and $g$ with $2\epsilon < \mathrm{inj}^\perp (g)$ (cp.~\eqref{eq_US} for the definition of $U_{S,g}(\epsilon)$).\medskip

For every map  $\epsilon \colon \RM\to (0,1)$ with  $2\epsilon\leq \mathrm{inj}^\perp$ we  introduce the map 
\begin{align}\adh_{\epsilon}\colon \RM \times [0,1]&\to \Riem,\nonumber\\ (g,t)&\longmapsto (1-t\eta_{\epsilon(g),g})g+t\eta_{\epsilon(g),g} (\exp^\perp_g)_*(\xi_{n-k}+ g|_S).\label{eq:Se}
\end{align}
 Note that by the choice of $\epsilon$ the metric $(\exp^\perp_g)_*(\xi_{n-k}+ g|_S)$ exists at least on $U_{S,g}(2\epsilon(g))$ and hence by definition of $\eta_{\epsilon(g),g}$ the image of $S_\epsilon$ really defines a smooth metric on $M$.
For $t=1$ this is exactly the deformation Ammann, Dahl and Humbert have used in \cite{ADH} to show that for fixed $g$  and for $\epsilon(g)$ small enough $\adh_{\epsilon}(g,1)\in \RM$.\medskip 

It is immediate to see that $ \adh_{\epsilon}(g,0)=g$ for all $g\in\RM$ and that 
 $\adh_{\epsilon}(g,1)$ has the desired half-flat structure on $U_{S,g}(\epsilon(g))$ as claimed in \eqref{A}.\medskip
 
 In the following we will prove  that $\hat{\epsilon}\colon \RM\to (0,1)$ can be chosen continuously such that $2\hat{\epsilon}\leq \mathrm{inj}^\perp$ and that for all continuous functions $\epsilon\colon \RM\to (0,1)$ with $\epsilon\leq \hat{\epsilon}$ we have 
\begin{enumerate}[(I)]
\item\label{I}  $\text{image} (\adh_{\epsilon})\subset \RM$ and
\item\label{II} $\adh_{\epsilon}$ is continuous.
                                \end{enumerate}
 
 Let
 \[\Riemf \define \{g\in \RM\ |\ \exists \epsilon'\in (0,1)\colon g|_{U_{S,g}(\epsilon')}= d\hat{r}^2+\hat{r}^2\sigma_{n-k-1}+g|_S\},\] as mentioned in Remark~\ref{rem:Rflat}, where $\hat{r}$ is the normal radial coordinate to $S$ induced by $g$. Then, $S_\epsilon(\Riemf\times [0,1])\subset \Riemf$. Hence, as soon as \eqref{I} and \eqref{II} are proven, we have established that $\adh_\epsilon$ is a homotopy from $\RM$ to $\Riemf$. By construction $S_\epsilon$ is then a homotopy inverse to the inclusion $\Riemf\hookrightarrow \RM$. This map will be the first part of the desired homotopy from $\RM$ to $\Rf$, cp. Remark~\ref{rem:Rflat}.\medskip 
 
   The continuity \eqref{II} directly follows from the continuous dependence of $\eta$, $\epsilon$ and $\exp^\perp$ on $g$. Hence, it remains to show \eqref{I}:
   
 \subsubsection{Proof of \eqref{I}}

   The proof relies on the fact that the modification, even though not $C^1$-small, happens only in a small tubular neighbourhood of the surgery sphere $S$.\medskip 
   
   First we need an auxiliary lemma similar to \cite[Lemma 3.1]{ADH}:
   
       \begin{lemma}\label{continuitydeformation}
   There are continuous maps $\hat{\mu}, C\colon \RM\to \mathbb R_+$ with $2\hat{\mu}\leq \mathrm{inj}^\perp$ such that for all $g\in \RM$ and all $\mu\leq \hat{\mu}(g)$ we have
   $$\|\hat{G}_g\|_{C^0(U_{S,g}(2\mu),g)} \leq C(g)\mu \qquad \|\nabla^g \hat{G}_g\|_{C^0(U_{S,g}(2\hat{\mu}(g)),g)}\leq C(g)$$
      where $\hat{G}_g\define g- (\exp_g^\perp)_*(\xi_{n-k}+g|_S)$.
  \end{lemma}

  \begin{proof}The proof locally mimics the proof of \cite[Lemma 3.1]{ADH} and then uses a covering argument: Let $g\in \RM$. Then, there is an open neighbourhood $U_g\subset \RM$ of $g$, an $R_g>0$ small enough and a $C_{1,g}>0$ such that $\|\nabla^h \hat{G}_h\|_{C^0(U_{S,g}(2R_g),h)}\leq C_{1,g}$ for all $h\in U_g$.  Note that $\hat{G}_g = 0$ on $S$. Hence, there is a $C_{2,g}>0$ such that $|\hat{G}_h(p)|\leq C_{2,g} r_h(p)$ for all $h\in U_g$ and $p\in U_{S,g}(2R_g)$ where $r_h$ is the radial distance function to $S$ w.r.t. $h$. We set $C_g\define \max\{C_{1,g}, C_{2,g}\}$.\smallskip 

  We note that $\{U_g\}_{g\in \RM}$ is an open cover of $\RM$. Since $\RM$ is a metric space, it is in particular paracompact \cite{Rud}. Hence, we have a partition of unity $\chi_{g}$ subordinated to this cover. We set 
  $\hat{\mu}\define \sum_{g\in \RM} R_g \chi_g$ and $C\define \sum_{g\in \RM} C_g\chi_g$. By construction these two maps are automatically continuous and fulfil the estimates of the Lemma.
  \end{proof}

\begin{lemma}\label{lem:existencedelta}
 For any $(g,t)\in\RM \times [0,1]$ there exists a positive number $\mu\define \mu(g,t)<  \mathrm{inj}^\perp(g)/2$ and an open neighbourhood $U\subset \RM\times [0,1]$ of $(g,t)$ such that for all $\mu'<\mu$ and all $(g',t')\in U$ the metric $\adh_{\mu'}(g',t')$ belongs to $\RM$.
\end{lemma}

\begin{proof}
The proof is obtained as the one of \cite[Lemmata 3.3 and 3.4]{ADH} in a parametrized way: Assume there is no such $\mu$ and $U$. Then there are sequences $\mu_i\to 0$, $g_i\to g$ in $\RM$ and $t_i\to t \in [0,1]$ as $i\to\infty$ such that there are harmonic spinors to $\hat{g}_i\define \adh_{\mu_i}(g_i,t_i)$, i.e., $\varphi_i\in\Gamma(\Sigma^{\hat{g}_i}M)$ satisfying $D^{\hat{g}_i}\varphi_i=0$ and $ \int_M|\varphi_i|^2\dvol_{\hat{g}_i}=1$.\medskip

Let $c(g)$ be such that $ c(g)^{-1} \Vert .\Vert_{C^j(g_0)}\leq \Vert .\Vert_{C^j(g)}\leq c(g) \Vert .\Vert_{C^j(g_0)}$ for all $(0,2)$-tensors and $j=0,1$.  Such a constant exists since $M$ is compact. Note that $c(g)$ can be chosen continuously in $g$.\medskip 

Using  $\mu_i\leq \hat{\mu}(g)$ for $i$ large enough, $g_i\to g$,  Lemma~\ref{continuitydeformation} and  that $\hat{g}_i=g_i-t_i\eta_{\mu_i, g_i}\hat{G}_i$ for $\hat{G}_i\define g_i- (\exp_{g_i}^\perp)_*(\xi_{n-k}+g_i|_S)$, we have for large enough $i$ that 
\begin{align}\nonumber
\|g-\hat{g}_i\|_{C^0(g)} &\leq \| g-g_i\|_{C^0(g)}+ t_i\|\eta_{\mu_i, g_i} \hat{G}_i\|_{C^0(U_{S,g_i}(2\mu_i),g)}\\  
&\leq c(g)\| g-g_i\|_{C^0(g_0)}+ c(g)c(g_i)^{-1}C(g_i)\mu_i\nonumber \\
&\leq c(g)\| g-g_i\|_{C^0(g_0)}+ 2C(g_i)\mu_i
,\label{metricdistance1} 
\end{align}
and similar
\begin{align}
 \|\nabla^{g}(g-\hat{g}_i)&\|_{C^0(g)} \leq  \Vert\nabla^g(g-g_i)\|_{C^0(g)}+   t_i\|\nabla^{g}(\eta_{\mu_i, g_
 i} \hat{G}_i)\|_{\mathcal{C}^0(U_{S,g_i}(2\mu_i),g)}\nonumber\\
 &  \leq c(g)\big(\| g-g_i\|_{C^1(g_0)} + c(g_i)^{-1} \|\eta_{\mu_i, g_
 i} \hat{G}_i\|_{\mathcal{C}^1(U_{S,g_i}(2\mu_i),g_i)})\nonumber\\
 &  \leq c(g) \| g-g_i\|_{C^1(g_0)} +2 C(g_i)(\mu_i +3).
 \label{metricdistance2}  
 \end{align} 
Since $C(g)$ depends continuously on $g$ by Lemma~\ref{continuitydeformation} and  $g_i\to g$, there is a constant $\hat{C}>0$ that bounds $C(g_i)$ for $i$ large enough.\medskip
 
 Since for any value of $i$ the spinor $\varphi_i$ belongs to a different spinor bundle,  we use the identification maps $\beta_g^{\hat{g}_i}$ from \eqref{spinorisonorm}  and $\Sigma^gM$ as a reference bundle: $\beta_g^{\hat{g}_i}\varphi_i\in \Gamma(\Sigma^gM)$.\medskip 
 
 Note that $\Vert \beta_g^{\hat{g}_i}\varphi_i\Vert^2_{L^2(g)}=\int_M |\varphi_i|^2 \dvol_{g}\to 1$ as $i\to \infty$. We proceed by showing via contradiction that the sequence $\beta_g^{\hat{g}_i}\varphi_i$ is bounded in $H^1(\Sigma^gM, g)$. For that suppose that 
$$\alpha_i\define \sqrt{\int_M {|\nabla^{g}(\beta_g^{g_i}\varphi_i)|^2_{g}}\dvol_g}$$ diverges for $i\to\infty$.  Let $\psi_i\define \alpha_i^{-1} \beta_{g}^{{\hat{g}_i}}\varphi_i$. Then, $
 {}^gD^{\hat{g}_i}\psi_i 
 =\alpha_i^{-1} \beta_{g}^{\hat{g}_i}D^{\hat{g}_i}\varphi_i=0$  by \eqref{Dg1g2}. 
  Using the Schr\"odinger-Lichnerowicz formula and \eqref{Dg1g2AB}, we obtain
 \begin{align*}
  1 &= \int_M |\nabla^{g}\psi_i|^2\dvol_g = \int_M{\left(|D^g\psi_i|^2-\frac{1}{4}\text{scal}_{g}|\psi_i|^2\right) \dvol_g} \\[3mm] 
  &\leq 2 \int_M{\left(|{}^{g}D^{\hat{g}_i}\psi_i|^2+|A^{g}_{\hat{g}_i}\nabla^{g}\psi_i|^2+|B_{\hat{g}_i}^{g}\psi_i|^2\right)\dvol_g}+\hat{s}\int_M{|\psi_i|^2 \dvol_g}\end{align*}
  where $\hat{s}\define \sup_M \text{scal}_g$. The first integral is estimated using \eqref{spinorisonorm}, \eqref{metricdistance1} and \eqref{metricdistance2}. Moreover,  $\int_M|\psi_i|^2\dvol_g=\alpha_i^{-2}\int_M|\varphi_i|^2\dvol_g$. Thus, we obtain with $\epsilon_i\define c(g)\Vert g-g_i\Vert_{C^\infty(g_0)}$ that 
  \begin{align*}
  1&\leq 2C(\epsilon_i + 2\hat{C} \mu_i)^2\underbrace{\int_{M}{|\nabla^{g}\psi_i|^2\dvol_{g}}}_{= 1}+2C(7\hat{C}+\epsilon_i)^2 \underbrace{\int_{M}|\psi_i|^2\dvol_g}_{= 2\alpha_i^{-2}} \\
  &\quad + \hat{s}\alpha_i^{-2}\int_M|\varphi_i|^2\dvol_g \to 0.\end{align*}
               
 This gives a contradiction and implies that $\beta_g^{\hat{g}_i}\varphi_i$ is bounded in $H^1(\Sigma^gM, g)$.\label{page_argH1} Hence, a subsequence  converges weakly in $H^1(\Sigma^gM,g)$ and strongly in $L^2(\Sigma^{g}M,g)$ to some $\varphi\in  \Gamma(\Sigma^gM)$ with $\Vert \varphi\Vert_{L^2(M,g)}=1$.\medskip 
 
 Fix $\mu>0$. For $i$ big enough $\mu_i<\mu$. Since $\eta_{\mu_i, g_i}\equiv 0$ on $M\setminus U_{S,g_i}(2\mu_i)$ and  $M\setminus U_{S,g}(3\mu)\subset M\setminus U_{S,g_i}(2\mu_i)$ for large $i$, the metrics $g_i$ and $\hat{g}_i=S_{\mu_i}(g_i,t_i)$ coincide on $M\setminus U_{S,g}(3\mu)$. Together with Lemma~\ref{lem:Schauder}  we see that  $\Vert\beta_{g}^{\hat{g}_i}\varphi_i\Vert_{C^2(M\setminus U_{S,g}(3\mu),g)}$ is bounded.\medskip

 Thus, by  Arz\'ela-Ascoli, Lemma~\ref{lem:Arzela}, the limit spinor $\varphi$ is in $C^1_\text{loc}(M\setminus U_{S,g}(3\mu),g)$ for $\mu >0$. Since the limit is the same for different $\mu$, the limit spinor satisfies the equation $D^g\varphi=0$ on $M\setminus S$ and $\Vert \varphi\Vert_{L^2(M,g)}=1$. Using  Lemma~\ref{lem:removalsingularityF}  $\varphi$ is a nonzero harmonic spinor on all of $(M,g)$ which gives the contradiction. 
  \end{proof}
        
\begin{corollary}\label{cor:deltaconstant} There is a continuous function 
  $\hat{\epsilon}\colon \RM\to (0,1]$ with $2\hat{\epsilon}< \mathrm{inj}^\perp$ and such that for all $g\in \RM$, $t\in [0,1]$ and all $\mu'<\hat{\epsilon}(g)$  the metric \[\adh_{\mu'}(g,t)= (1-t\eta_{\mu',g})g+ t\eta_{\mu',g} (\exp_g^\perp)_* \left(\xi_{n-k}+\sigma_k\right)\] belongs to $\RM$.
\end{corollary}

\begin{proof}
For each $(g,t)\in \RM\times [0,1]$ let $\mu(g,t)$ resp. $U(g,t)$ be the $\mu$ resp. $U$ obtained in Lemma~\ref{lem:existencedelta}. 
Let now $\chi_{g,t}$ be a partition of unity subordinated to the open cover $\cup_{(g,t)\in \RM \times [0,1]} U(g,t)= \RM\times [0,1]$. We set \[\hat{\mu}(g,t) \define \sum_{(g',t')\in \RM\times [0,1]} \mu(g',t')\chi_{g',t'}(g,t).\] 
By construction, $\hat{\mu}$ is continuous and positive and $\hat{\mu} (g,t)\leq \text{sup}_{(g,t)\in U(g',t')} \mu(g',t')$. Thus,  $\hat{\mu} (g,t)$ fulfils Lemma~\ref{lem:existencedelta} for an appropriate $U$. Let $\hat{\epsilon}\colon \RM\to (0,1]$ be defined as $g\mapsto \min\{ \min_{t\in [0,1]} \hat{\mu}(g,t), 1\}$. Then $2\hat{\epsilon}< \mathrm{inj}^\perp$. Since $\hat{\mu}$ is continuous and $[0,1]$ is compact, the image of $\hat{\epsilon}$ is really a subset of $(0,1]$ and $\hat{\epsilon}$ is again continuous.
\end{proof}

\begin{remark}
 In case we would state Lemma~\ref{lem:existencedelta} only  for $\mu'=\mu$, the result would directly follow from the original versions in \cite{ADH} and that invertibility is an open property. But since the function $\hat{\epsilon}$  in Corollary~\ref{cor:deltaconstant} needs to be specified later and since we do not know yet how small $\hat{\epsilon}$ needs to be, we prove here everything for all positive $\mu'$ less than a threshold. 
\end{remark}

\subsection{Standardizing}\label{sec_stand}

First we will construct the diffeomorphisms for \eqref{B} on page~\pageref{B}.\medskip 

For that recall from page~\pageref{B} that we choose a continuous family of smooth monotonically increasing functions $\{ a_{\epsilon}\colon [0,2]\to [0,2]\}_{\epsilon\in (0,1]}$ with $a_{\epsilon}(r)=r$ for $r\in (0, \frac{\epsilon}{4})$, $a_{\epsilon}(1)=\epsilon$, $a_\epsilon|_{[3/2,2]}=\mathrm{id}$ and $a_1\equiv \mathrm{id}$ and define
\[ \Rf \define \{ g\in \RM\ |\,  \exists \delta\in (0,1]\colon\! g|_{D_1\setminus S} = a_{\delta}'(r)^2dr^2 + a_{\delta}(r)^2\sigma_{n-k-1} +\! g|_S\}.\]
The $\delta$ in the last definition is uniquely determined by $g$, since $a_\delta(1)=\delta$. Hence, we have a continuous function $\delta\colon \Rf\to (0,1]$. Since $\Rf\subset \RM$ is a closed subset (see \eqref{B} on page~\pageref{eq_Rf}), this function is extendable to a continuous positive function on all of $\RM$ by Tietze's extension theorem. In the following, we choose such an extension and call it 
\begin{align}\label{def_delta}\delta\colon \RM\to (0,1],\end{align} as well.

\begin{lemma}\label{lem:stand}
 There  are continuous maps $\Delta\colon \RM\to (0,\infty)$, $\Theta\colon \RM\to \Riem$ and $\Phi\colon \RM\times [0,1]\times M\to M$ such that for all $g\in \RM$, $t\in [0,1]$
 \begin{enumerate}[(i)]
  \item $\Phi_{g,t}\define \Phi(g,t,.)\in \mathrm{Diff}(M)$ and $\Phi_{g,0}=\mathrm{id}$,
  \item $U_{S,g}(\Delta(g))\!\Subset\! D_2$ and $(\exp_g^\perp)^{-1}$ is a well-defined diffeomorphism on $U_{S,g}(\Delta(g))$,
  \item $\Phi_{g,t}|_{M\setminus D_2}=\mathrm{id}$ and $\Phi_{g,t}|_{U_{S,g}(\Delta(g))}= \exp_{(1-t)g+t\Theta(g)}^\perp\circ (\exp_g^\perp)^{-1}$,
  \item $D_u=U_{S, \Theta(g)}(a_{\delta(g)}(u))$ for all $u\in [0,1]$ with $\delta\colon \RM\to (0,1]$ as introduced above.
 \end{enumerate}
In particular, $(\Phi_{g,t})_*g\in \Rf$ for all $g\in \Rf$ and all $t\in [0,1]$.
 \end{lemma}

 \begin{proof} Let $\lambda\colon M\to M$ be a smooth function with $\lambda|_{D_1}\equiv 1$ and $\lambda|_{M\setminus D_2}\equiv 0$. Then the map 
\begin{align*}
 \Theta\colon \RM&\to \text{Riem}(M),\\
 g&\mapsto (1-\lambda)g+\lambda\left( a_{\delta(g)}'(r)^2dr^2 + a_{\delta(g)}(r)^2\sigma_{n-k-1}+g|_S\right),
\end{align*}
is continuous and fulfils (iv). Note that we do not claim that $\Theta(g)$ has invertible Dirac operator in general. For $g\in \Rf$ we have $\Theta(g)|_{D_1}=g|_{D_1}$.\medskip  
 
 Let $\mathrm{inj}^\perp\colon \RM\to \mathbb R_+$ be  again the normal injectivity radius to $S$ and define $\tilde{\Delta}(g)\define \min\{\sup\{ a\ |\ U_{S,g}(a)\subset D_{3/2}\}, \min_{\tau\in [0,1]} \mathrm{inj}^\perp ((1-\tau)g+\tau \Theta(g))\}$.
For $g\in \Rf$ we have $\mathrm{inj}^\perp(g)> \delta(g)$. Since $g\in \Rf$ coincides with $\Theta(g)$ on $D_1$, this implies $\mathrm{inj}^\perp ((1-\tau)g+\tau \Theta(g))>\delta (g)$. Thus, for all $g\in \Rf$ it is $\tilde{\Delta}|_{\Rf}> \delta$. Hence, there is a continuous function $\Delta\colon \RM\to (0,\infty)$ with  \[\Delta (g) <
 \tilde{\Delta}(g)\quad \text{and}\quad \Delta|_{\Rf}> \delta.\]
 By construction $U_{S,g}(\Delta(g))\subset D_{3/2}\Subset D_2$ and the first property in (ii) is fulfilled. Moreover, $U_{S,g}(\Delta(g))\supset U_{S,g}(\delta(g))=D_1$ for $g\in \Rf$.\medskip

If we have a $\Phi$ fulfilling (i) and (iii), then the construction of $\Theta$ and $\Delta$ ensures that for $g\in \Rf$ we have $\Phi_{g,t}|_{D_1=U_{S,g}(\delta(g))}=\mathrm{id}$. Thus, $(\Phi_{g,t})_*g\in \Rf$ for all $(g,t)\in \Rf\times [0,1]$ and (ii) is fulfilled.\medskip 
  
  Thus, it remains to construct $\Phi$ such that (i) and (iii) are fulfilled:  Let $\Delta'\colon \RM \to (0,\infty)$ be a continuous function with $\Delta < \Delta' < \tilde{\Delta}$. Let $\hat{\eta}\colon \RM \times M\to [0,1]$ be continuous such that $\hat{\eta}_g\define \hat{\eta} (g,.)$ is smooth, $\hat{\eta}_g|_{U_{S,g}(\Delta(g))}=1$ and $\hat{\eta}_g|_{M\setminus U_{S,g}(\Delta'(g))}=0$.  We set 
  \begin{align*}
   X_g(p,t)\define \hat{\eta}_g(p)\frac{d}{d\tau}|_{\tau=t} \left( \exp_{(1-\tau)g+\tau\Theta(g)}^\perp\circ (\exp_g^\perp)^{-1}\right)(p) + \partial_t. 
  \end{align*}
  This is well-defined: by the choice of $\Delta'$ and $\hat{\eta}$ we have that $\mathrm{inj}^\perp ((1-\tau)g+\tau \Theta(g))> \Delta'(g)$ and hence, $(\exp_g^\perp)^{-1} (U_{S,g}(\Delta'(g)))\subset \text{domain} (\exp^\perp_{(1-\tau)g+\tau \Theta(g)})$.\medskip
  
  By construction $X_g$ is a smooth vector field on $M\times [0,1]$ that depends continuously on $\RM$. We note that $X_g \equiv \partial_t$ on $(M\setminus D_2) \times [0,1]$ and $X_g= \frac{d}{dt} \exp_{(1-t)g+t\Theta(g)}^\perp\circ (\exp_g^\perp)^{-1}+\partial_t$ on  $U_{S,g}(\Delta(g)) \times [0,1]$ (which is needed to give the prescribed parts of $\Phi_{g,t}$ in (iii)). Hence, as in \cite[Thm.~2.4.2]{Wall} 
   $X_g$ defines diffeomorphisms $\Phi_{g,t}$ with the desired properties.
 \end{proof}

\begin{proof}[Proof of Proposition~\ref{prop:Rf}]\label{pr:Rf}
Let $\delta, \Delta \colon \RM \to (0,\infty)$ be the continuous functions from above. Let $\hat{\epsilon}\colon \RM\to (0,1]$ be as in Corollary~\ref{cor:deltaconstant}. We define $\epsilon \colon \RM\to (0,1]$ by $\epsilon(g)\define \min \{\hat{\epsilon}(g),  \delta(g)/8, \sup\{ r\in [0,1]\ |\ D_r\subset U_{S,g}(\Delta(g))\}\}$. Since also $\hat{\epsilon}$ is continuous, $\epsilon$ is continuous as well. Moreover, $2\epsilon\leq 2\hat{\epsilon}<\mathrm{inj}^\perp$.\medskip 
 
 Let $T_\epsilon\colon \RM\times [0,1] \to \Riem$ be defined by 
 \begin{align}
  (g,t)\mapsto \left\{ \begin{matrix}
                        S_\epsilon(g, 3t) & t\in [0, \frac{1}{3}]\\
                        (\Phi_{g, 3t-1})_*  S_\epsilon(g,1)& t\in (\frac{1}{3}, \frac{2}{3}]\\
                        \left( a_{ (3-3t)\delta(g) + (3t-2)\epsilon(g)} \circ a_{\delta(g)}^{-1}\right)_*  (\Phi_{g,1})_* S_\epsilon(g,1)& t\in (\frac{2}{3},1 ]
                       \end{matrix}
\right.
\label{map:T}
 \end{align}
 with $S_\epsilon$ as defined in \eqref{eq:Se} and $\Phi$ as in Lemma~\ref{lem:stand}. Moreover, the maps $a_\delta$ from Section~\ref{sec_stand} are viewed as maps on $M$ by extending $a_\delta$ constantly in the $S^{n-k-1}\times S^k$-direction, by identity on $M\setminus D_2$. Then, $a_\delta\in \mathrm{Diff}(M)$.\medskip 
 
 The map $T_\epsilon$ is continuous by the continuity of the involved maps and since  $S_\epsilon(g,1)= (\Phi_{g,0})_*S_\epsilon(g,1)$ and $(\Phi_{g, 1})_*S_\epsilon(g,1) = \left( a_{\delta(g)}\circ a_{\delta(g)}^{-1}\right)_*  (\Phi_{g, 1})_*S_\epsilon(g,1)$.\medskip 
 
 From  Section~\ref{sec:half}, see Corollary~\ref{cor:deltaconstant}, it follows that for $t\in [0,1/3]$ the map $T_\epsilon$ maps into $\RM$. For bigger $t$ this follows since we only pullback by diffeomorphisms. Hence, $\text{Image}(T_\epsilon)\subset \RM$.\medskip 
 
We note that $T_\epsilon(g,1/3)=S_\epsilon(g,1)=(\exp_g^\perp)_*(\xi_{n-k}+g|_S)$  on $U_{S,g}(\epsilon(g))$ and $\mathrm{inj}^\perp T_\epsilon (g, 1/3)\geq 2\epsilon(g)$. With the choice of $\epsilon$ we have $D_{\epsilon(g)}\subset U_{S,g}(\Delta(g))$. Thus, with Lemma~\ref{lem:stand}(iii) we have on $D_{\epsilon(g)}$ that
\begin{align*}T_\epsilon(g,2/3)=&(\Phi_{g,1})_* (\exp_g^\perp)_* (\xi_{n-k}+g|_S)=(\exp_{\Theta(g)}^\perp)_* (\xi_{n-k}+g|_S)\\
=&a_{\delta(g)}'(r)^2 dr^2+a_{\delta(g)}(r)^2\sigma_{n-k-1} +g|_S\overset{\epsilon \leq \delta/8}{=} dr^2 + r^2 \sigma_{n-k} +g|_S.
\end{align*}
Hence,
\begin{align*} D_{\epsilon(g)} &\overset{a_{\delta(g)}^{-1}|_{D_{\epsilon(g)}} =\mathrm{id}}{\to} D_{\epsilon(g)} \overset{a_{\epsilon(g)}}{\to} D_1 \\
dr^2+r^2\sigma_{n-k-1}+g|_S& \overset{ \left(a_{\epsilon(g)}\circ a_{\delta(g)}^{-1}\right)_*}{ \longmapsto} a_{\epsilon(g)}'(r)^2dr^2+a_{\epsilon(g)}(r)^2\sigma_{n-k-1}+g|_S.
\end{align*}
Hence, we have $T_\epsilon (g,1)\in \Rf$ with $\delta(T_\epsilon (g,1))=\epsilon(g)$.\medskip 

In order to see that $T_\epsilon$ gives the desired homotopy $\Rf\cong \RM$, it remains to check that $T_\epsilon (g,t)\in \Rf$ for all $g\in \Rf$ and $t\in [0,1]$: For $t\leq 1/3$ this follows from $\epsilon\leq \delta/8$ and $g=dr^2+r^2\sigma_{n-k-1}+g|_S$ on $r\leq \delta(g)/4$ by definition of $\Rf$.  For $t\in (1/3,2/3]$ this follows from the last statement of Lemma~\ref{lem:stand}. For $t\geq 2/3$ this follows since: By  Lemma~\ref{lem:stand} we have  $T_\epsilon(g,2/3)=a_{\delta(g)}'(r)^2 dr^2+a_{\delta(g)}(r)^2\sigma_{n-k-1} +g|_S$ on $U_{S,g}(\delta(g))\subset U_{S,g}(\Delta(g))$ and, hence, on $D_1\subset a_{(3-3t)\delta(g) + (3t-2)\epsilon(g)}(D_{\delta(g)})$ (which follows from $(3-3t)\delta(g) + (3t-2)\epsilon(g)\leq \delta(g)$).
\begin{align*} T_\epsilon(g, t)=& \left( a_{ (3-3t)\delta(g) + (3t-2)\epsilon(g)} \circ a_{\delta(g)}^{-1}\right)_* (a_{\delta(g)}'(r)^2 dr^2+a_{\delta(g)}(r)^2\sigma_{n-k-1} +g|_S)\\
 =&a_{(3-3t)\delta(g) + (3t-2)\epsilon(g)}'(r)^2 dr^2+a_{(3-3t)\delta(g) + (3t-2)\epsilon(g)}(r)^2\sigma_{n-k-1} +g|_S\qedhere
\end{align*}
\end{proof}

  \section{Embedding \texorpdfstring{$\Rf$}{Rf} into \texorpdfstring{$\Rc$}{Rc}---Proof of Proposition~\ref{prop:existencerho}}
   \label{pr:existencerho}

We define the functions appearing in the Definition of \eqref{eq:Y}, see also  Figure~\ref{fig:fcts}: Let 
 \begin{align}
  F(r)&\define \begin{cases}\frac{1}{r}& 0<r<1/2\\ 1 & 3/4\leq r,\end{cases}\nonumber\\
  f_{\rho}(r)\define f(\rho, r)&\define 
   \begin{cases}                                                                                                                                                                                                                                                                                                  r&  0<r \leq \rho\\
   1 &   2\rho\leq r \end{cases}  \label{eq:defFfeta}\\
 \eta_{\rho}(r)\define \eta(\rho,r)&\define
  \begin{cases}                                                                                                                                                                                                                                                                                  0\quad &  0<r\leq \rho\\
  1\quad &   2\rho\leq r   \end{cases} \nonumber
  \end{align}
(with $ \rho \in (0,\tfrac{1}{4})$, $r\in (0,2]$) be  such that these functions are smooth in $r$ and as families continuous in  $\rho$. Using the diffeomorphism (given by $\exp_{g_0}^\perp$) $D_2\setminus S\cong (0,2]\times S^{n-k-1}\times S^k$, we extend these functions constantly in all other variables and by identity on $M\setminus D_2$ to obtain continuous functions
\begin{align*}
 F\colon   M\setminus S\to \mathbb R,\quad
 f\colon  (0,\tfrac{1}{4})\times M\setminus S\to \mathbb R,\quad
 \eta\colon (0,\tfrac{1}{4})\times M\setminus S\to \mathbb R. 
\end{align*}
The functions $F, f_\rho, \eta_\rho$ are smooth and still continuous in $\rho$ when viewed as functions on $M\setminus S$.\medskip

In order to prove Proposition~\ref{prop:existencerho}, we first see  in Proposition~\ref{prop:existencerho_old} below, that a function $\rho$ exists such that the image of $y_\rho$ has invertible Dirac operator:\medskip 

 \begin{proposition}\label{prop:existencerho_old}
  Let $\delta\colon \Rf\to (0,1)$  be as in Proposition~\ref{prop:Rf}. Then there exists a continuous function $\hat{\rho}\colon \Rf\to (0,1]$ with $0<\hat{\rho}<\delta/32$ such that for all $\rho\leq \hat{\rho}(g)$ and all $g\in \Rf$ the metric \begin{align*} y_{\rho}(g)=\! \left\{\! 
\begin{matrix}\hspace{4cm}           g \hspace{5.3cm} \text{ on } M\!\setminus\!D_1    \\[0.2cm]                                                                                                                                                                                                                                                
F^2\!\left(a_{\delta(g)}'(r)^2dr^2+a_{\delta(g)}(r)^2\sigma_{n-k-1}+
  \ f_{\rho}^2\left(\eta_{\rho}g|_S+(1-\eta_{\rho})\sigma_k\right)\right)\!  \text{ on } D_1\!\setminus\!S\end{matrix}
\right. 
\end{align*}
 has an invertible Dirac operator. In particular $(y_\rho(g), \rho)\in \Rc$.
 \end{proposition}
 
 The proof uses the ideas and methods of \cite[Proposition 3.2]{ADH}, except that here we perform no surgery but look at the blown-up manifold $M\setminus S$, and we want a continuous blow-up parameter $\rho$. 
 
\begin{proof}We note first that $(M\setminus S, y_\rho(g))$ is a manifold with cylindrical end on $D_{\min\{\rho, \delta(g)/4\}}$ and, thus, it is complete. In particular, as soon as we no that this manifold has invertible Dirac operator, we have $(y_\rho(g), \rho)\in \Rc$.\medskip  

In order to prove the existence of the function $\hat{\rho}$ it is enough to show, that given any $g\in \Rf$ we can find a number $\rho\in (0,1)$ such that $y_{\rho'}(\hat{g})$ has invertible Dirac operator for all $\rho'\leq \rho$ and all $\hat{g}$ near enough to $g$. Then the proposition follows by a covering argument as in Lemma~\ref{continuitydeformation}.\medskip 

Note that $\rho'\leq \rho\leq \hat{\rho}\leq \delta /32$ and $y_{\rho'}(g)=\frac{dr^2}{r^2}+\sigma_{n-k-1}+\sigma_k$ for $r\leq \rho'$. Hence, by Lemma~\ref{lem:spec_cyl} the Dirac operator on $y_{\rho'}(g)$ has no essential spectrum, and invertibility can only be prevented by the existence of $L^2$-harmonic spinors.\medskip 

The strategy to show the non-existence of harmonic spinors for $\rho$ small enough and  metrics near enough to $g$ is by contradiction, i.e., by proving that for all $\rho_i\to 0$ and  all $g_i\in \Rf\to g$  any sequence of $D^{y_{\rho_i}(g_i)}$-harmonic spinors in $L^2(M\setminus S, y_{\rho_i}(g_i))$  converges (after using appropriate identification maps, see \eqref{eq:beta})   to a $D^g$-harmonic spinor in $L^2(M,g)$: We abbreviate $y_i\define y_{\rho_i}(g_i)$. Let $\psi_{i}$ be a $D^{y_i}$-harmonic spinor in $L^2(M\setminus S, y_i)$. Since $y_i= F^2 g_i$ on $M\setminus D_{2\rho_i}$, by \eqref{conformalchange} the spinor $F^{\frac{n-1}{2}}\psi_i$ is $D^{g_i}$-harmonic on  this set. We prove next that for all $\mu>0$ the sequence $\varphi_i\define \beta_g^{g_i}(F^{\frac{n-1}{2}}\psi_{i})$ converges in $C^1_{\text{loc}}(M\setminus U_{S,g}(\mu), g)$ to some nonvanishing $\varphi\in L^2(M\setminus S,g)$.\medskip

For that we first provide a weighted $L^2$-estimate for any $D^{y_{\rho} (\hat{g})}$-harmonic spinor $\psi\in L^2(M\setminus S, y_{\rho} (\hat{g})$ away from $S$ for some $\rho< \delta(g)/32$  and all $\hat{g}\in \Rf$ with  $\Vert\hat{g}-g\Vert_{C^\infty(M,g_0)}<\epsilon$. Here, the $\epsilon>0$ is chosen such that $\delta(\hat{g})\leq 2\delta(g)$ and $|.|^2_{\hat{g}}\leq 2|.|^2_{g}$ on one-forms.\medskip 

We choose $u\in(2\rho, \delta(g)/16)$. Recall that $a_{\delta(\hat{g})}(r)=r$ for $r\leq \delta(\hat{g})/4$. Hence, together with the choice of $u$ and $2u\leq \frac{\delta(g)}{8}\leq \frac{\delta(\hat{g})}{4}$, the metric $y_\rho(\hat{g})$ on $D_{2u}$ is isometric to 
 $\frac{dr^2}{r^2} + \sigma_{n-k-1} + h_{\rho}(r)$, where for each $r$ the $h_{\rho}(r)$ is  a metric on $S^k$.\medskip 
 
 We define a smooth cut-off function
$$\chi(r)=\begin{cases}
   1,\qquad &\text{on }D_u,\\
   0, &\text{on } M\setminus D_{2u},
  \end{cases}
$$
 such that $|d\chi|_g\leq 2/u$ on $D_{2u}\setminus D_u$. This is possible, since $U_{S,g}(z)=D_z$ for all $z\leq \delta(g)/4$.\medskip
 
   As the square of the Dirac operator for product manifolds $(M_1\times M_2,g_1+ g_2)$ splits, we have that $D^{y_\rho(\hat{g})}$, with domain  restricted to smooth spinors with support in $D_{2u}$ (like $\chi\psi$), has at least the spectral gap of the standard round metric $\sigma_{n-k-1}$, i.e., 
  \begin{equation}\label{Rayleighquotient}\frac{\int_{D_{2u}}|D^{y_\rho(\hat{g})}(\chi\psi)|^2\dvol_{y_\rho(\hat{g})}}{\int_{D_{2u}}|\chi\psi|^2\dvol_{y_\rho(\hat{g})}}\geq \frac{(n-k-1)^2}{4} \geq \frac{1}{4}.\end{equation}
  
With $D^{y_\rho(\hat{g})}\psi=0$ we obtain $|D^{y_\rho(\hat{g})}(\chi\psi)|=|d\chi|_{y_\rho(\hat{g})}|\psi|$. Moreover, since  $y_\rho(\hat{g})=F^2\hat{g}$ on $M\setminus D_{2\rho}$, $\text{supp}\, d\chi=D_{2u}\setminus D_u\subset M\setminus D_{2\rho}$ and $a_{\delta(\hat{g})}(r)=r$ for $r\leq 2u\leq \frac{\delta(g)}{8}\leq \frac{\delta(\hat{g})}{4}$,  we have 
$$|d\chi|^2_{y_\rho(\hat{g})}=a_{\delta(\hat{g})}(r)^2|d\chi|^2_{\hat{g}}\leq r^2 2|d\chi|^2_{g}
\leq\frac{8r^2}{u^2}$$ on $D_{2u}\setminus D_u$. Hence, together with \eqref{conformalchange} and $F(r)=1/r$ for $r\leq 1/2$ we estimate
\begin{align}
  \int_{D_{2u}}|D^{y_\rho(\hat{g})}(\chi\psi)&|^2\dvol_{y_\rho(\hat{g})}  \leq  \frac{8}{u^2}\int_{D_{2u}\setminus D_u}r^{2-n}|\psi|^2\dvol_{\hat{g}}\nonumber \\ 
  = \frac{8}{u^2}\int_{D_{2u}\setminus D_u}&r|F^{\frac{n-1}{2}}\psi|^2\dvol_{\hat{g}} \overset{r\leq 2u}{\leq }
  \frac{16}{u}\int_{D_{2u}\setminus D_u}|F^{\frac{n-1}{2}}\psi|^2\dvol_{\hat{g}}. \label{numeratorRayleigh}
 \end{align} 
 For the denominator of the Rayleigh quotient \eqref{Rayleighquotient} we estimate
\begin{align}
  \int_{D_{2u}}|\chi \psi|^2\dvol_{y_\rho(\hat{g})}&  \geq \int_{D_u\setminus D_{2\rho}}|\psi|^2\dvol_{y_\rho(\hat{g})} \nonumber \\ 
     = \int_{D_u\setminus D_{2\rho}}r^{-1}| &F^{\frac{n-1}{2}}\psi|^2\dvol_{\hat{g}} 
  \geq 
  \frac{1}{u}\int_{D_u\setminus D_{2\rho}}|F^{\frac{n-1}{2}}\psi|^2\dvol_{\hat{g}}.\label{denominatorRayleigh}
 \end{align}
 Inserting \eqref{numeratorRayleigh} and \eqref{denominatorRayleigh} into \eqref{Rayleighquotient}, we obtain
\begin{equation}\label{L2decay}
 \int_{D_u\setminus D_{2\rho}}|F^{\frac{n-1}{2}}\psi|^2\dvol_{\hat{g}}\leq 64 \int_{D_{2u}\setminus D_u}|F^{\frac{n-1}{2}}\psi|^2\dvol_{\hat{g}}.
\end{equation}
In particular, estimate \eqref{L2decay} says that the $L^2(\hat{g})$-norm of the $D^{y_\rho(\hat{g})}$-harmonic spinor $\psi$  tends to avoid the cylindrical end.\medskip

Let now  $\varphi \define  F^\frac{n-1}{2}\psi\in\Gamma(\Sigma^{\hat{g}}(M\setminus D_{2\rho}))$. By \eqref{conformalchange} $D^{\hat{g}}\phi=0$ on $M\setminus D_{2\rho}$. For any choice of $\mu\in (2\rho,u)$ we notice that \eqref{L2decay} implies 
$$\int_{D_u\setminus D_{\mu}}
 |\varphi|^2\dvol_{\hat{g}}\leq 64\int_{M\setminus D_u}|\varphi|^2\dvol_{\hat{g}},$$ 
 and finally 
 \begin{align}\label{L2bounded} \int_{M\setminus D_{\mu}}|\varphi|^2\dvol_{\hat{g}} \leq 
 & (1+64)\int_{M\setminus D_u}|\varphi|^2\dvol_{\hat{g}}.
              \end{align}\medskip

We now return to the sequence $y_i\define y_{\rho_i}(g_i)$  with the $D^{y_i}$-harmonic spinors $\psi_{i}$ from the beginning.  We assume that those spinors are normalized such that  $\int_{M\setminus D_{u}}|F^\frac{n-1}{2}\psi_{i}|^2\dvol_{g_i}=1$ for a fixed $u\in (0, \delta(g)/16)$. 
Then, \eqref{L2bounded} implies that for $\rho_i\to 0$ (then  $u\in (2\rho_i, \delta(g)/16)$ for $i$ large enough) the sequence of $D^{g_i}$-harmonic spinors $\{F^\frac{n-1}{2}\psi_i\}_i$ remains bounded in $L^2(\Sigma^{g_i}(M\setminus D_{\mu}),g_i)$ for all $\mu\in (0,u)$. Since $g_i\to g$, also $\{\beta_g^{g_i}(F^\frac{n-1}{2}\psi_i)\}_i$ remains bounded in $L^2(\Sigma^{g}(M\setminus D_{\mu}),g)$ for all $\mu\in (0,u)$.\medskip 

Then, since $\mathrm{scal}_g$ is bounded the same  arguments as in the last paragraphs of the  proof of  Lemma~\ref{lem:existencedelta} on p.~\pageref{page_argH1} give that $\varphi_{i}$ converges in $C^1_{\text{loc}}(M\setminus D_\mu,g)$ to a $D^{g}$-harmonic spinor $\varphi$ on $M\setminus D_\mu$. By \eqref{L2bounded} it is ${\left\| \varphi\right\|}^2_{L^2(\Sigma^{g}(M\setminus D_\mu),g)}\leq 65$ for all $\mu\in (0,u)$. Letting $\mu\to 0$ we obtain $\phi\in C^1_{\text{loc}}(M\setminus S,g)$ with 
$$1={\left\| \varphi\right\|}^2_{L^2(\Sigma^{g}(M\setminus D_u),g)}\leq{\left\| \varphi\right\|}^2_{L^2(\Sigma^{g}(M\setminus S),g)}\leq 65.$$  By 
Lemma~\ref{lem:removalsingularityF}  $\varphi$ is then a  strong harmonic spinor on all of $(M,g)$ which gives the contradiction. 
\end{proof}

\begin{proof}[Proof of Proposition~\ref{prop:existencerho}]
 In Proposition~\ref{prop:existencerho_old} we already obtained a continuous function $\hat{\rho}\colon \Rf\to (0,\tfrac{1}{2})$ with $\hat{\rho}\leq \delta/32$ such that $\Upsilon_\rho(\Rf)\subset \Rc$  is fulfilled for all $\rho\leq \hat{\rho}$.\medskip
 
At first we prove that $\Upsilon_\rho$ is a homeomorphism onto its image:\smallskip 

Let $(h,s)\in \Upsilon_\rho(\Rf)$. The definition of $\Upsilon_\rho$ implies that  $h|_{r=1}$ has the form
  \[h|_{r=1}= \hat{\delta}(h)^2\sigma_{n-k-1}+\text{res}(h).\]
 From this form we can directly read of continuous maps $\text{res}\colon\! y_\rho(\Rf) \to \text{Riem}(S)$ and $\hat{\delta}\colon y_\rho(\Rf)\! \to (0,1]$. Note that $\hat{\delta}(y_{\rho(g)}(g))=a_{\delta(g)}(1)=\delta(g)$. With these functions we define the map 
 \begin{align*}
 \kappa\colon \Upsilon_\rho(\Rf)\to& \Riem\\
  (h,s)\mapsto& \left\{ \begin{matrix}
                        h & \text{on }M\setminus D_1\\
                        a_{\hat{\delta}(h)}'(r)^2dr^2 + a_{\hat{\delta}(h)}(r)^2\sigma_{n-k-1}+\text{res}(h) & \text{on } D_1.
                       \end{matrix}
 \right.
 \end{align*}
Note that $h=y_{\rho(g)}(g)$ for some $g\in \Rf$ and hence $\kappa(h,s)|_{D_1}=g|_{D_1}$. Since by construction $g|_{M\setminus D_{3/4}}=y_{\rho(g)}(g)|_{M\setminus D_{3/4}}$, $\kappa(y_{\rho(g)}(g),s)=g$ really gives a smooth metric on $M$. Hence, $\kappa$  is the left-inverse to $\Upsilon_\rho$ and maps onto $\Rf$.\medskip
 
 Continuity of $\Upsilon_\rho$ directly follows from the continuity of $\rho, \delta$ and the definitions of $F$,$f_\rho$, $\eta_\rho$ in \eqref{eq:defFfeta} and $a_\delta$ in Section~\ref{sec_stand}. The continuity of $\kappa$ follows by the continuity of $\hat{\delta}$ and $\text{res}$. This implies (i).\medskip
 
  It remains to see that the image of $\Upsilon_\rho$ is closed in $\Rc$: We recall that $\Upsilon_\rho(\Rf)$ is really a subset of $\Rc$ by Proposition~\ref{prop:existencerho_old}.
 Let $\Upsilon_\rho(g_i)=(y_{\rho(g_i)}(g_i),\rho(g_i))\in \Upsilon_\rho(\Rf)$ converge to $(g,s)\in \Rc$. Then, $\rho(g_i)\to s>0$. Moreover, the radius $\hat{\delta}(y_{\rho(g_i)}(g_i))=\delta(g_i)$ of the $S^{n-k-1}$-factor at $r=1$ needs to converge to some $\delta\in (0,1]$ and $g_i|_S=\text{res}(y_{\rho(g_i)}(g_i))$ converges to some $\tilde{h}\in \text{Riem}(S^k)$. Hence, $g_i\to \check{g}$ with $\check{g}=g$ on $M\setminus D_1$ and $\check{g}=a_\delta'(r)^2dr^2+a_\delta(r)^2\sigma_{n-k-1}+\tilde{h}$ on $D_1$. This implies $\check{g}\in \Rf$. By continuity of $\rho$ this implies $s=\rho(\check{g})$ and, hence,  $\Upsilon(\check{g})=(g,s)$ the claim follows.
\end{proof}

\section{The grafting}\label{sec:cup}

First we prove the result on the gluing of cylindrical manifolds:

\begin{proof}[Proof of Lemma~\ref{lem:L}]\label{pr:L}
 Let $(g,s)\in \Rc$. We show that there is an $L>0$ and an open neighbourhood $U_{(g,s)}\subset \Rc$ of $(g,s)$ such that for all $(\hat{g},\hat{s})\in U_{(g,s)}$, $\hat{L}\geq L$ and $h\in \text{R}_{c(\hat{g}, \hat{s})}(N)$ the Dirac operator to the glued metric has invertible Dirac operator. The rest then again follows by a covering argument as in Lemma~\ref{continuitydeformation}.\medskip
 
 We  prove this claim by contradiction: Assume that there are sequences $L_i\in \mathbb R_+$ with $L_i\to \infty$, $(g_i,s_i)\in \Rc$ with $(g_i,s_i)\to (g,s)$ and $h_i\in \text{R}_{c(g_i, s_i)}$  such that the Dirac operator to the glued metric is not invertible. Let $(Z_i,G_i)$ denote the glued manifold.\medskip

 By construction of $\Rc$, for $i$ large enough $g_i=g=\frac{dr^2}{r^2}+\sigma_{n-k-1}+\sigma_k$ for $r<\frac{s}{2}$.\medskip 
 
 We note that any element in the essential spectrum of $D^{G_i}$ needs to come from one of the ends of $N$ that was not glued to $M$ and, thus, from the essential spectrum of $D^{h_i}$. Hence, the zero in the spectrum of $D^{G_i}$, if existent, is an eigenvalue.\medskip
 
 Let $\phi_i$ be a $D^{G_i}$-harmonic spinor with $\Vert \phi_i\Vert_{L^2(Z_i,G_i)}=1$. For $L_i> 2j$ let $\chi_{i,j}\colon Z_i\to [0,1]$ be smooth functions such that 
 $\chi_{i,j}=1$ on $M\setminus \{r<s_i\}\cup [0,j]\times S^{n-k-1}\times S^k  \subset Z_i$, $\chi_{i,j}=0$ on $[2j, L_i]\times S^{n-k-1}\times S^k\cup N\setminus \{\hat{r}<2\}  \subset Z_i$ and $|d\chi_{i,j}|_{G_i}\leq 2/j$.  Since $G_i=g_i$ on $\{\chi_{i,j}\neq 0\}$ and $G_i=h_i$ on $\{\chi_{i,j}\neq 1\}$, we obtain
 \begin{align}\label{eq:D1}
  \frac{\Vert D^{g_i}(\chi_{i,j} \phi_i)\Vert_{L^2(g_i)}}{\Vert \chi_{i,j}\phi_i\Vert_{L^2(g_i)}}=  \frac{\Vert d\chi_{i,j}\cdot \phi_i\Vert_{L^2(G_i)}}{\Vert \chi_{i,j}\phi_i\Vert_{L^2(g_i)}}\leq& \frac{2}{j\Vert \chi_{i,j}\phi_i\Vert_{L^2(g_i)}}\quad \text{ and }\\ 
  \label{eq:D2} 
      \frac{\Vert D^{h_i}((1-\chi_{i,j}) \phi_i)\Vert_{L^2(h_i)}}{\Vert (1-\chi_{i,j})\phi_i\Vert_{L^2(h_i)}}=
    \frac{\Vert d\chi_{i,j} \phi_i\Vert_{L^2(G_i)}}{\Vert (1-\chi_{i,j})\phi_i\Vert_{L^2(h_i)}}\leq& \frac{2}{j\Vert (1-\chi_{i,j})\phi_i\Vert_{L^2(h_i)}}.
 \end{align}
 Moreover, $1=\Vert \phi_i\Vert_{L^2(G_i)}\leq \Vert \chi_{i,j}\phi_i\Vert_{L^2(g_i)}+\Vert (1-\chi_{i,j})\phi_i\Vert_{L^2(h_i)}$  which implies that one of the quantities $\limsup_{i\to \infty} \Vert \chi_{i,j}\phi_i\Vert_{L^2(g_i)}$ and $\limsup_{i\to \infty} \Vert (1-\chi_{i,j})\phi_i\Vert_{L^2(h_i)}$ has to be $\geq 1/2$.\medskip
 
 Let $a>0$ be the spectral gap of $D^g$.  For $i$ large enough the infimum of the spectrum of $(D^{g_i})^2$ is bigger than $a^2/4$ by Lemma~\ref{lem:spec_cyl}(ii), and $c(g_i,s_i)\geq c(g,s)/2$. Let $j\geq 8/\min \{ a, c(g,s)\}$. Then, one of the right sides of \eqref{eq:D1} and \eqref{eq:D2} is smaller than $\min\{ a/2, c(g,s)/2\}$. That means, that the spectral gap of $D^{g_i}$ is smaller than $a/2$ or the one of $D^{h_i}$ is smaller than $c(g_i,s_i)$ which gives the contradiction.
\end{proof}

At the end we want to glue (in the sense of the last lemma) every $(g,s)\in \Rc$ with the following 'blown-up torpedo metric': We define on $T_{n-k}\setminus (\{0\}\times S^k)= (\mathbb R^{n-k}\setminus \{0\})\times S^k$ the metric 
 \begin{align*}
 y_{\tilde{\rho}}(h_{torp}^{n,k})\define F(\tilde{r})^2(h_{n-k} + f_{\tilde{\rho}}(\tilde{r})^2\sigma_k),
\end{align*}

with $h_{torp}^{n,k}$ as in Subsection~\ref{sec:torp}. As a map $y_{\tilde{\rho}}$ was originally defined in Proposition~\ref{prop:existencerho_old} but for metrics in $\Rf$. But since this metric behaves very similar we call it $y_{\tilde{\rho}}(h_{torp}^{n,k})$ nevertheless.    Note that $y_{\tilde{\rho}}(h_{torp}^{n,k})$ blows up the 'origin' of the torpedo while leaving the original cylindrical end untouched, i.e. we obtain a manifold diffeomorphic to $(0,\infty) \times S^{n-k-1}\times S^k$ with two cylindrical ends -- the original one on $\tilde{r}\geq 2$ with metric $d\tilde{r}^2+\sigma_{n-k-1}+\sigma_k$ and the new one  on $\tilde{r}\leq {\tilde{\rho}}$ with metric $\frac{d\tilde{r}^2}{\tilde{r}^2}+\sigma_{n-k-1}+\sigma_k$. In particular, zero cannot be in the essential spectrum of the correponding Dirac operator. Moreover,
\begin{align*} y_{\tilde{\rho}}(h_{torp}^{n,k})=& F(\tilde{r})^2(d\tilde{r}^2+ w(\tilde{r})^2\sigma_{n-k-1}+f_{\tilde{\rho}}(\tilde{r})^2\sigma_k)\\
 =& F(\tilde{r})^2w(\tilde{r})^2(\sigma_{n-k-1}+ w(\tilde{r})^{-2}(d\tilde{r}^2+f_{\tilde{\rho}}(\tilde{r})^2\sigma_k))\\
 =& F(\tilde{r})^2w(\tilde{r})^2(\sigma_{n-k-1}+ \text{complete metric on} (0,\infty)\times S^k)
\end{align*}

The conformal factor equals $1$ near $\tilde{r}=0$ and near infinity. Thus, by \eqref{conformalchange} and invertibility of the Dirac operator on $\sigma_{n-k-1}$, which implies the invertibility of the product metric $F^{-2}w^{-2} y_{\tilde{\rho}}(h_{torp}^{n,k})$, we see that $y_{\tilde{\rho}}(h_{torp}^{n,k})$ has invertible Dirac operator for all $\tilde{\rho}\in (0,1)$.\medskip 

In the following, we choose $\rho\define \hat{\rho}\colon \Rf\to (0,1)$ for the $\hat{\rho}$ from Proposition~\ref{prop:existencerho}.
\medskip 

For Proposition~\ref{prop:cup} we want to continuously deform $(g,s)\in \Rc$ into an element in $\Upsilon_\rho(\Rf)$ for $\rho$ to be chosen. The endpoint of this deformation will be a gluing of $(g,s)$ with $y_\rho(h_{torp}^{n,k})$ in the sense of last lemma. That means in particular, that we need to identify this glued metric as the image of $y_\rho$ of an element in $\Rf$, see (b) and (d) below. For that we need not only do the abstract gluing but to specify the diffeomorphism of the glued together manifold $Z$ with $M$. The resulting metric on $M$ will be called $g_{\mathrm{tor}}$ later on, see \eqref{eq_gtor} (we drop the dependence on $s$ in the notation). 

For the deformation in between we glue $(g,s)$ together with an interpolation of the standard cylinder with $y_\rho(h_{torp}^{n,k})$ in a continuous way. For that we in particular need that this interpolation always has invertible Dirac operator, which will done by a combination of an interpolation on the cylindrical part, see (a) below, the last Lemma and diffeomorphisms of $M$, see (c), that ensure that the end of the deformation lands in $\Upsilon_\rho(\Rf)$.\medskip

\paragraph{\textbf{(a)}} Let $u\colon (0,\infty)\to \mathbb R$ be a smooth monotonically increasing function with $u(\tilde{r})=\tilde{r}$ for $\tilde{r}\geq 2$ and $u(\tilde{r})=\ln \tilde{r}$ for $\tilde{r}\leq 1$. Changing the coordinate $\tilde{r}$ into $u$ we obtain  for the interpolation of $y_\rho(h_{torp}^{n,k})$  with the standard metric on $\mathbb R\times S^{n-k-1}\times S^k$:

\begin{lemma}\label{lem:Lint} There is a $\zeta>0$
such that  
 \[ G_{\rho,t}\define (1-t)\left(du^2 + \sigma_{n-k-1} +\sigma_k\right)+t y_\rho(h_{torp}^{n,k})\in R_{\zeta}(\mathbb R \times S^{n-k-1}\times S^k)\] 
 for all $t\in [0,1]$ and $\rho\in (0,1)$.
\end{lemma}

\begin{proof}
On $u\in \mathbb R \setminus (\ln \rho, 2)$ the resulting metric is 
$du^2+\sigma_{n-k-1}+\sigma_k$.  In general, we have 
  \begin{align*} G_{\rho,t} =&(1-t)(du^2 + \sigma_{n-k-1} +\sigma_k)+t y_\rho(h_{torp}^{n,k}) \\
   =&((1-t)+tF(\tilde{r}(u))^2\tilde{r}'(u)^2)du^2 \\
   &+\underbrace{(1-t+tF(\tilde{r}(u))^2w(\tilde{r}(u))^2)}_{=:Q(u,t)>0} \sigma_{n-k-1}+(1-t+tF(\tilde{r}(u))^2f_\rho(\tilde{r}(u))^2)\sigma_k\\
   = &Q(u,t) \left(\sigma_{n-k-1} + (\text{complete metric on }\mathbb R\times S^k)\right).
  \end{align*}
Thus, the metric $G_{\rho,t}$ is conformal to a product metric on $S^{n-k-1}\times (\mathbb R\times S^k)$ where $S^{n-k-1}$ is equipped with $\sigma_{n-k-1}$ and the conformal factor is equal to $1$ outside a compact subset. Hence, the Dirac operator to $G_{\rho,t}$ is invertible for all $t\in [0,1]$ and $\rho\in (0,1)$. In particular implies the spectral gap of $\sigma_{n-k-1}$ and \eqref{conformalchange} that $\inf \mathrm{spec} (D^{G_{\rho,t}})^2\geq \frac{\inf_{u\in \mathbb R} (Q(u,t))^{-1}}{\sup_{u\in \mathbb R} Q(u,t)}\frac{(n-k-1)^2}{4}=:\zeta (t)$.\medskip 

We set $\zeta\define \inf_{t\in [0,1]} \zeta(t)$. Since the function $\zeta(t)$ depends continuously on $t$, $\zeta$ is a positive number with $G_{\rho,t}\in \text{R}_{\zeta}(\mathbb R\times S^{n-k-1}\times S^k)$.
\end{proof}

\paragraph{\textbf{(b)}} We choose $L$ to be the function obtained in Lemma~\ref{lem:L} for $(N,h)=(T_{n,k}, h_{\text{torp}}^{n,k})$ with $c(g,s)=\frac{(n-k-1)^2}{4}$.  Then for $(g,s)\in \Rc$ we have by Lemma~\ref{lem:L} that
$Z((g,s), h_{\text{torp}}^{n,k}, L(g,s))$  has invertible Dirac operator. Note that the glued manifold $Z$ is diffeomorphic to $M$. We will soon  view $Z((g,s), h_{\text{torp}}^{n-k}, L(g,s))$ as a metric on $M$ and even an element in $\Rf$, see (d) below.
\medskip 

\paragraph{\textbf{(c)}} Let $\Psi\colon (0,1]\times [0,2)\to [0,2)$ be a smooth function
such that $\Psi_u\define \Psi(u,.)$ is monotonically increasing, $\Psi_u(r)=r$ for $r\in (2-u,2)$, $\Psi_u(r)=r/u$ for $r\leq u$ and $\Psi_1=\mathrm{id}$. Extending $\Psi_u$ constant in $S^{n-k-1}\times S^k$-direction and by identity on $M\setminus D_2$ this gives a continuous one parameter family $\hat{\Psi}_u \in \mathrm{Diff}(M)$. Note that \begin{align}\label{eq_cyl_diff} (\hat{\Psi}_u)_*\left(\left(\frac{dr^2}{r^2}+\sigma_{n-k-1}+\sigma_k\right)|_{D_u\setminus S}\right)=\left(\frac{dr^2}{r^2}+\sigma_{n-k-1}+\sigma_k\right)|_{D_1\setminus S}.\end{align}

\paragraph{\textbf{(d)}} We view  $Z((g,s), h_{\text{torp}}^{n-k}, L(g,s))$ as  a metric on $M$ in the follows way: 
It equals $(\Psi_{se^{-L(g,s)-1}})_*g$ on $M\setminus D_{\tilde{v}}$ for $\tilde{v}=\Psi_{se^{-L(g,s)-1}}(se^{-L(g,s)})$, i.e., $M\setminus D_{\tilde{v}}$ contains the  $M\setminus \{r\leq s\}\cup [0,L(g,s)]\times S^{n-k-1}\times S^k$ of $Z$. On  $D_{\tilde{v}}$ the metric $Z((g,s), h_{\text{torp}}^{n-k}, L(g,s))$ thus should just give the $T_{n,k}\setminus \{\tilde{r} \geq 2\}$ part of $Z$ and should be equal to $(\Psi_{se^{-L(g,s)-1}})_*g$ on $D_{\tilde{v}}\setminus D_v$ for $v=\Psi_{se^{-L(g,s)-1}}(s)$ (the $L(g,s)$-long cylindrical part): For that let $\tilde{\Psi}\colon \Rc\times [0,\infty)\to [0,\infty)$ be  a continuous function, such that $\tilde{\Psi}_{(g,s)}\define \tilde{\Psi} (( g,s),.)$ is a smooth monotonically increasing function with 
\[ \left\{ \begin{matrix}
          \tilde{\Psi}_{(g,s)}(\tilde{r})=\tilde{r}  \hspace{5cm}\text{for } \tilde{r}\in (0,1)\hspace{1.4cm}\\
            \tilde{\Psi}_{(g,s)}( \tilde{r})= \hat{\Psi}_{se^{-L(g,s)-1}}(se^{2-\tilde{r}}) \hspace{2.4cm} \text{for }\tilde{r}\in (2, 2+L(g,s)).
           \end{matrix}
\right. \]
We view $\tilde{\Psi}_{(g,s)}$ as diffeomorphism on $T_{n,k}$ by extending it constantly perpendicular to the radial direction. Note that by construction $\Psi_{(g,s)}$ depends continuously on $(g,s)\in \Rc$. The choice of $\tilde{\Psi}$ is such that
\begin{align} 
\nonumber (\tilde{\Psi}_{(g,s)})_*h_{\text{torp}}^{n-k}&=h_{\text{torp}}^{n-k} \text{ on }\tilde{r}\leq 1\\
\nonumber (\tilde{\Psi}_{(g,s)})_*(h_{\text{torp}}^{n-k}|_{\{2\leq  \tilde{r} \leq 2+L(g,s)\}})&=  (\hat{\Psi}_{se^{-L(g,s)-1}})_*(\frac{dr^2}{r^2}+\sigma_{n+k+1}+\sigma_k)\label{eq_equiv_cy}\\&=(\hat{\Psi}_{se^{-L(g,s)-1}})_*(g)\\ &\quad \nonumber\text{ on }
\{\tilde{v}\leq r=\tilde{\Psi}_{(g,s)}(\tilde{r})\leq v\}.\end{align}
\medskip

Altogether this gives a continuous map 
\begin{align}\label{eq_gtor}\Rc\mapsto \Rf,\, (g,s)\!\mapsto\! g_{\mathrm{tor}}\define \left\{ \begin{matrix}
(\Psi_{se^{-L(g,s)-1}})_*g &\text{on } M\!\setminus\! D_{\tilde{v}}\\
                                         (\tilde{\Psi}_{(g,s)})_* h_{\text{torp}}^{n-k} & \text{on } D_{\tilde{v}}.
                                        \end{matrix}
 \right.\end{align}
By \eqref{eq_equiv_cy} the image is a smooth metric. Note that $g_{\mathrm{tor}}=g$ on $M\setminus D_2$,  the set $D_2\setminus D_1$ contains (among other parts ) the $L(g,s)$-long gluing cylinder and on $D_1$ the metric equals $dr^2+r^2\sigma_{n-k-1}+\sigma_k$.\medskip

Moreover, for $0<v< \tilde{v} <2$ let $\kappa_{v, \tilde{v}}\colon M\to M$ be smooth with $\kappa_{v, \tilde{v}}=1$ for $r\leq \tilde{v}$ and $\kappa_{v, \tilde{v}}\equiv 0$ on $M\setminus \{r\leq v\}$ and such that the map depends continuously on $v$, $\tilde{v}$.\medskip 

\begin{figure}
\centering

\begin{tikzpicture}[scale=0.54]

\draw (19,1.5) node { $t=0$};
\draw[line width=1] (0,0) .. controls (-2,1) and (-2,2) .. (0,3) .. controls (2,4) and (3,4) .. (4,3.7) .. controls (5, 3.6) and (5.5, 4) .. (6, 4)  ..     controls (9.2,4.3) and (9.4,3.1) .. (9.5,3) .. controls (10.5, 1.5) and (10.8, 2) .. (11,2) -- (18,2); 
\draw[line width=1] (18,1) -- (11,1) .. controls (10.8,1) and (10.5, 1.5) .. (9.5,0) .. controls (9.4, 0) and (9.2, -1.3) .. (6,-1) .. controls (5.5,-1) and (5,-0.6) .. (4, -0.7) .. controls (3,-1) and (2,-1) .. (0,0);
\draw[line width=1] (0,2) .. controls (1,1.5) .. (2,2);
\draw[line width=1] (0.4,1.8) .. controls (1,2) .. (1.6,1.8);
\begin{scope}[shift={(2,-1)}]
\draw[line width=1] (0,2) .. controls (1,1.5) .. (2,2);
\draw[line width=1] (0.4,1.8) .. controls (1,2) .. (1.6,1.8);
\end{scope}

\draw[dashed] (6.5,5.3) node[above] {\small $r=2$ } -- (6.5,-7.5);
\draw[dashed] (9.6,4.5) node[above] {\small  $r=1$ } -- (9.6,-7.5);

\draw[dashed] (11,5.3) node[above] {\small  $r=s$ } -- (11,-3.5);
\draw[color=blue, -> ]  (11,0.8) .. controls  (11,-2) and (7.9,-1) .. (7.9,-4); 
\draw[dashed] (7.9,-4) -- (7.9, -7.6) node[below, xshift=-3] {\small $r=v$};

\draw[dashed] (13.5,4.5) node[above, xshift=-5, yshift=1] {\small  $r=se^{-L(g,s)}$ } -- (13.5,-3.5);

\draw[color=blue, -> ]  (13.5,0.8) .. controls  (13.5,-2) and (8.5,-1) .. (8.5,-4); 
\draw[dashed] (8.5,-4) -- (8.5, -8.4) node[below, yshift=2, xshift=2] {\small $r=\tilde{v}$};
\draw[dashed] (8.5,-8.9) -- (8.5, -12.5);
\draw[dashed] (12,-9.2)  node[above, xshift=4] {\small $r=\rho(g_{{tor}})$} -- (12, -12.5);

\draw[dashed] (15,5) node[xshift=27pt, yshift=14] {\small  $r=se^{-L(g,s)-1}$ } -- (15,-3.5);

\draw[color=blue, -> ]  (15,0.8) .. controls  (15,-2) and (9.6,-1) .. (9.6,-4);

\draw [decorate,decoration={brace,amplitude=4pt},xshift=-4pt,yshift=0pt]
 (11.3,2.5) -- (13.5,2.5) node [midway,yshift=13pt] {\footnotesize 
$L(g,s)$};

\draw[->,color=blue]  (18.3,1) .. controls (19, -1.5) ..  (18.3,-4);
\draw[color=blue] (17,-1.7) node  {\small diffeotopy};
\draw[color=blue] (17,-2.4) node  {\small $\Psi_{se^{-L(g,s)-1}}$};

\draw[color=blue] (3,-1.7) node {\small identity on $M\setminus D_2$};

\begin{scope}[shift={(0,-6.5)}]
\draw (19.26,1.5) node { $t\in [\frac{1}{4}, \frac{1}{2}]$};
\draw[line width=1] (0,0) .. controls (-2,1) and (-2,2) .. (0,3) .. controls (2,4) and (3,4) .. (4,3.7) .. controls (5, 3.6) and (5.5, 4) .. (6, 4)  ..     controls (6.8,4.3) and (6.9,3.1) .. (7,2.7) .. controls (7.2, 1.5) and (7.5, 2) .. (8,2) -- (18,2); 
\draw[line width=1] (18,1) -- (8,1) .. controls (7.5,1) and (7.2, 1.5) .. (7,0.3) .. controls (6.9, 0) and (6.8, -1.3) .. (6,-1) .. controls (5.5,-1) and (5,-0.6) .. (4, -0.7) .. controls (3,-1) and (2,-1) .. (0,0);
\draw[line width=1] (0,2) .. controls (1,1.5) .. (2,2);
\draw[line width=1] (0.4,1.8) .. controls (1,2) .. (1.6,1.8);
\begin{scope}[shift={(2,-1)}]
\draw[line width=1] (0,2) .. controls (1,1.5) .. (2,2);
\draw[line width=1] (0.4,1.8) .. controls (1,2) .. (1.6,1.8);
\end{scope}

\draw[->,color=blue]  (18.3,1) .. controls (19, -1.5) ..  (18.3,-4);
\draw[color=blue] (16.3,-1.4) node  {\small interpolation};
\draw[color=blue] (16.3,-2) node  {\small between $t=\frac{1}{2}$};
\draw[color=blue] (16.3,-2.6) node  {\small and $y_{\rho(g_{\mathrm{tor}})}({g}_{\mathrm{tor}})$};

\end{scope}

\begin{scope}[shift={(0,-13)}]
\draw (19,1.5) node { $t=1$};
\draw[line width=1] (0,0) .. controls (-2,1) and (-2,2) .. (0,3) .. controls (2,4) and (3,4) .. (4,3.7) .. controls (5, 3.6) and (5.5, 4) .. (6, 4)  ..     controls (6.8,4.3) and (6.9,3.1) .. (7,2.7) .. controls (7.2, 1.5) and (7.5, 2) .. (8,2) -- (9,2) .. controls (9.5,2) and(10,1.7) .. (10.5,1.7) .. controls (11,1.7) and (11.5,2) .. (12, 2)-- (18,2); 
\draw[line width=1] (18,1) -- (12,1) .. controls (11.5,1) and (11, 1.3) .. (10.5, 1.3) .. controls (10,1.3) and (9.5, 1) .. (9,1) -- (8,1) .. controls (7.5,1) and (7.2, 1.5) .. (7,0.3) .. controls (6.9, 0) and (6.8, -1.3) .. (6,-1) .. controls (5.5,-1) and (5,-0.6) .. (4, -0.7) .. controls (3,-1) and (2,-1) .. (0,0);

\draw[line width=1] (0,2) .. controls (1,1.5) .. (2,2);
\draw[line width=1] (0.4,1.8) .. controls (1,2) .. (1.6,1.8);
\begin{scope}[shift={(2,-1)}]
\draw[line width=1] (0,2) .. controls (1,1.5) .. (2,2);
\draw[line width=1] (0.4,1.8) .. controls (1,2) .. (1.6,1.8);
\end{scope}

\draw [decorate,decoration={brace,amplitude=4pt},xshift=-4pt,yshift=0pt]
 (12,0.5) -- (8.8,0.5) node [midway,yshift=-13pt, xshift=25pt] {\footnotesize only here the metric differs from $t=\frac{1}{2}$};

\end{scope}

\end{tikzpicture}
\caption{The metric components of $\Xi_{\mathrm{gr}}$: $g$ at time $t=0$, $(\Psi_{se^{-L(g,s)-1}})_*g$ at $t\in [1/4,1/2]$ and $y_{{\rho}(g_{\mathrm{tor}})}(g_{\mathrm{tor}})$ at $t=1$. For $t<1/4$ this is only a pullback with a diffeomorphism. For $t> 1/2$ the map $\Xi_{\mathrm{gr}}$ only changes the metric on $[\rho(g_{\mathrm{tor}}), \tilde{v}]$.}
\label{fig:cup}
\end{figure}

With these notations and choices we obtain: 

\begin{lemma} \label{lem:cup}
 Let  ${\rho}\define \hat{\rho}\colon \Rf\to (0,1]$ as in Proposition~\ref{prop:existencerho}, $L\colon \Rc\to (0,\infty)$, $\Psi\colon (0,1]\times [0,2)\to [0,2)$ and $\hat{\Psi}_u\in \mathrm{Diff}(M)$ ($u\in (0,1]$) as chosen above. Then the map, see Figure~\ref{fig:cup},
 \begin{align*}
  \Xi_{\mathrm{gr}}\colon& \Rc \times [0,1] \to \Rc\\
  \nonumber ((g,s),t)&\mapsto \left\{ \begin{matrix}
       \left((\hat{\Psi}_{1-4t+4tse^{-L(g,s)-1}})_* g,\ s\right) & \text{for }t\in [0,\frac{1}{4})\\[0.2cm]
       \left( (\hat{\Psi}_{se^{-L(g,s)-1}})_* g,\ (4t-1)\rho(g_{\mathrm{tor}})  +(2-4t)s \right) & \text{for }t\in [\frac{1}{4},\frac{1}{2})\\[0.2cm]
      \bigg(\left(1-(2t-1)\kappa_{v,\tilde{v}}\right)(\hat{\Psi}_{se^{-L(g,s)-1}})_*(g) \qquad \\
       \qquad +(2t-1)\kappa_{v,\tilde{v}}y_{{\rho}(g_{\mathrm{tor}})} (g_{\mathrm{tor}}),\ \rho (g_{\mathrm{tor}}))\bigg) & \text{for }t\in [\frac{1}{2},1],
                               \end{matrix}
 \right.
 \end{align*}
 where $v=\Psi_{se^{-L(g,s)-1}}(s)$, $\tilde{v}=\Psi_{se^{-L(g,s)-1}}(se^{-L(g,s)})$ and $g_{\mathrm{tor}}$ is obtained from $(g,s)$ as above, is 
 \begin{enumerate}[(i)]
  \item well-defined and continuous
\item $\Xi_{\mathrm{gr}}(.,0)=\mathrm{id}$
\item  $\Xi_{\mathrm{gr}}((g,s),1)= \Upsilon_\rho(g_{\mathrm{tor}}) \in  \Upsilon_\rho(\Rf)$ for all $(g,s)\in \Rc$.
\end{enumerate}
($g_{\mathrm{tor}}$ was defined in \eqref{eq_gtor}.)
\end{lemma}

\begin{proof}(ii) follows directly by from $\hat{\Psi}_1=\mathrm{id}$.\medskip

(iii) We have $y_{\rho(g_{\mathrm{tor}})}(g_{\mathrm{tor}})= (\hat{\Psi}_{se^{-L(g,s)-1}})_*(g) $ on $M\setminus D_{\tilde{v}}$. Thus,  $\Xi_{\mathrm{gr}}((g,s),1)=\Upsilon_\rho(g_{\mathrm{tor}})$.\medskip 

For (i) we note that  $(\tilde{v}, v)\subset (1,2)$. For $t\in [0,1/4)$ the map $\Xi_{\mathrm{gr}}(.,.)$ is just a diffeotopy in the first component starting with $(\hat{\Psi}_1=\mathrm{id})_*g=g$ and such that at $t=1/4$ the set $\{\tilde{v}\leq r\leq v\}\subset M\setminus S$ is a cylindrical part of length $\int_{se^{-L(g,s)}}^s \frac{dr}{r}=L(g,s)$. Moreover, on $r\leq 1$ we have $\Xi_{\mathrm{gr}}((g,s),t))= (\frac{dr^2}{r^2}+\sigma_{n-k-1}+\sigma_k,s)$ by \eqref{eq_cyl_diff} and $\Xi_{\mathrm{gr}}((g,s),t))= \left(\frac{\hat{\Psi}_{se^{-L(g,s)-1}}'(r)^2dr^2}{\hat{\Psi}_{se^{-L(g,s)-1}}(r)^2}+\sigma_{n-k-1}+\sigma_k,s\right)$ on $r\leq v$. In particular, for $t\leq 1/4$ the image of $\Xi_{\mathrm{gr}}$ really lies in $\Rc$.\medskip 

 For $t\in [1/4,1/2)$ the metric is not changed but only the second component is moved to $\rho (g_{\mathrm{tor}})$ which is the second entry of $\Upsilon_{\rho}(g_{\mathrm{tor}})$. This is possible since $(\hat{\Psi}_{se^{-L(g,s)-1}})_*g$ has a cylindrical end w.r.t. $r$ for $r\leq 1$ and since $\rho(g_{\mathrm{tor}})\leq 1$.\medskip 
 
 Let now $t\in [1/2,1]$:  We note that by the choice of $\kappa$  the first component of $\Xi_{\mathrm{gr}}((g,s),t)$  equals $(\hat{\Psi}_{se^{-L(g,s)-1}})_*g$ on $M\setminus D_v$,
 $(\hat{\Psi}_{se^{-L(g,s)-1}})_*g=g_{\mathrm{tor}}$ on $D_{\tilde{v}}\setminus D_v
 $ and $y_{\rho (g_\mathrm{tor})}(g_{\mathrm{tor}})=\frac{dr^2}{r^2}+\sigma_{n-k-1}+\sigma_j$ on $D_{\rho(g_{\mathrm{tor}})}$. Moreover, by \eqref{eq_cyl_diff}  on  $D_{\tilde{v}}$ 
 \begin{align*} \left(1-(2t-1)\kappa_{v, \tilde{v}}\right)(\hat{\Psi}_{se^{-L(g,s)-1}})_* g
       +(2t-1)\kappa_{v, \tilde{v}}y_{\rho(g_{\mathrm{tor}})}(g_{\mathrm{tor}})\\
    = (\tilde{\Psi}_{(g,s)})_*\left((1-(2t-1))\left(\frac{dr^2}{r^2}+\sigma_{n-k-1}+\sigma_k\right)+  (2t-1)y_\rho(h_{\text{torp}}^{n-k})\right)
       \end{align*}
       equals $G_{\rho(g_{\mathrm{tor}}), 2t-1}$ 
        as in Lemma~\ref{lem:Lint} but in other coordinates. 
        Hence, by Lemma~\ref{lem:L} and the choice of $L$, we have that the first component of $\Xi_{\mathrm{gr}}((g,s),t)$ has invertible Dirac operator. Hence, $\Xi_{\mathrm{gr}}((g,s),t) \in \Rc$. This establishes that the map is well-defined. Continuity directly follow by the construction.
\end{proof}

\section{Proof of Theorem~\ref{thm:main}}\label{sec:mainthm}

Theorem~\ref{thm:main} is obtained by standard bordism arguments from Proposition~\ref{prop:main2}. We will give them here for the sake of completeness.\medskip 

Let $n=3$. Then, \cite[VII Thm. 3]{Kirby} implies that $M\sqcup S^3$ bounds a cobordism that only consists of $2$-handles. Hence, $S^3$ can be obtained from $M$ via surgeries of codimension $2$ only. Hence, for each of theses surgeries $k$ equals $1$ and, hence, fulfils the assumptions to Proposition~\ref{prop:main2}. Hence, $\mathcal{R}^{\text{inv}}(M)\cong \mathcal{R}^{\text{inv}}(S^3)$.\medskip

Let now $n>3$ and let $W$ be a spin cobordism from $M$ to $\widetilde{M}$. We can simplify  that $W$ until it is connected and simply connected by doing $0$ and $1$ dimensional surgeries (possible since $W$ is spin). Hence, by \cite[VIII Prop. 3.1]{Kos} $\widetilde{M}$ can be obtained from $M$ via finitely many surgeries of codimension $2\leq n-k\leq n-1$. Then Proposition~\ref{prop:main2} implies Theorem~\ref{thm:main}.

\begin{appendix}
\section{Torpedo metric} \label{app}
Let $h_{m}$ be a Riemannian metric on $\mathbb R^{m}$, $m\geq 2$, such that in standard spherical coordinates  $h_m =d\tilde{r}^2+w(\tilde{r})^2\sigma_{m-1}$ with radial coordinate $\tilde{r}$, $w(\tilde{r})=\tilde{r}$ for $\tilde{r}\leq 1$  and $w(\tilde{r})=1$ for $\tilde{r}\geq 2$. Then, the Dirac operator to $h_m$ is invertible:\medskip 

Since this is a metric with cylindrical end whose link is the standard sphere with invertible Dirac operator, the essential spectrum of $D^{h_m}$ does not contain $0$. Hence, only $L^2$-harmonic spinor can prevent the invertibility of $D^{h_m}$. Assume $\phi\in L^2(\Sigma^{h_m}\mathbb R^m)$ with $D^{h_m}\phi=0$. Let $D^S$ be the Dirac operator $D^{\sigma_{m-1}}$ on $S^{m-1}$ when $m$ is odd and $\text{diag}(D^S, -D^S)$ otherwise. Let $\psi_i$ be an orthonormal basis of eigenspinors. Let $D^S\psi_i=\lambda_i\psi_i$. Then we can expand $\phi=\sum_{i\in \mathbb N} \alpha_i(\tilde{r})\psi_i$. 

By \cite[3.6]{BGM} the Dirac operator is given by
\[-\partial_{\tilde{r}}\cdot D^{h_m}=\frac{1}{w(\tilde{r})}D^S +\frac{m-1}{2} \frac{w'(\tilde{r})}{w(\tilde{r})} +\nabla^{h_m}_{\partial_{\tilde{r}}}.\]
 
 Thus, $D^{h_m}\phi=0$ implies 
 \[ \lambda_i \frac{1}{w(\tilde{r})}\alpha_i(\tilde{r})+ \frac{m-1}{2} \frac{w'(\tilde{r})}{w(\tilde{r})} \alpha_i(\tilde{r})+ \alpha_i'(\tilde{r})=0\]
 for all $i\in \mathbb N$ and hence, $\alpha_i(\tilde{r})= c_i w(\tilde{r})^{-\frac{m-1}{2}} e^{-\int \frac{\lambda_i}{w(\tilde{r})} d\tilde{r}}$ for some $c_i\in \mathbb C$. Since $\lambda_i\neq 0$, $w(\tilde{r})=1$ for large $\tilde{r}$ and $\phi\in L^2$, it is $c_i=0$ or $\lambda_i>0$. Since $w(\tilde{r})=\tilde{r}$ for $\tilde{r}\leq 1$, it is $\alpha_i(\tilde{r})=c_i w(\tilde{r})^{-\lambda_i-\frac{m-1}{2}}$. Thus, in order to have no singularity at $\tilde{r}=0$ we need $c_i=0$ or $\lambda_i\leq -\frac{m-1}{2}$. Altogether this implies $c_i=0$ for all $i$ and hence $\phi=0$.
\end{appendix}

\end{document}